\documentclass[11pt,reqno,tbtags,a4paper]{amsart}
\usepackage{amssymb}
\usepackage{mathabx}  
\usepackage{xpunctuate}
\usepackage{url}
\usepackage{enumitem} 
\usepackage[square,numbers]{natbib}
\bibpunct[, ]{[}{]}{;}{n}{,}{,}
 
\usepackage{xcolor}

\title[Large fringe trees]
{Large fringe trees for random trees with given vertex degrees}

\date{7 April, 2026}  

\author{Gabriel Berzunza Ojeda}
\address{Department of Mathematical Sciences, University of Liverpool, United
  Kingdom} 
\email{gabriel.berzunza-ojeda@liverpool.ac.uk} 
\urladdr{https://www.liverpool.ac.uk/mathematical-sciences/staff/gabriel-berzunza-ojeda/}

\author{Cecilia Holmgren} 
\address{Department of Mathematics, Uppsala University, Sweden} 
\email{cecilia.holmgren@math.uu.se} 
\urladdr{https://katalog.uu.se/empinfo/?id=N5-824}

\author{Svante Janson}
\thanks{Funded by Knut and Alice Wallenberg Foundation; 
Ragnar Söderberg’s Foundation; Swedish Research Council}
\address{Department of Mathematics, Uppsala University, 
Sweden}
\email{svante.janson@math.uu.se}
\newcommand\urladdrx[1]{{\urladdr{\def~{{\tiny$\sim$}}#1}}}
\urladdrx{http://www2.math.uu.se/~svante/}

\subjclass[2020]{} 

\numberwithin{equation}{section}

\renewcommand\le{\leqslant}
\renewcommand\ge{\geqslant}

\allowdisplaybreaks

\theoremstyle{plain}
\newtheorem{theorem}{Theorem}[section]
\newtheorem{lemma}[theorem]{Lemma}

\newtheorem{corollary}[theorem]{Corollary}
\newtheorem{condition}[theorem]{Condition}

\theoremstyle{definition}

\newtheorem{exampleqqq}[theorem]{Example}
\newenvironment{example}{\begin{exampleqqq}}
  {\hfill\qedsymbol\end{exampleqqq}}

\newtheorem{remarkqqq}[theorem]{Remark}
\newenvironment{remark}{\begin{remarkqqq}}
  {\hfill$\triangle$\end{remarkqqq}}

\newtheorem*{acks}{Acknowledgements}

\theoremstyle{remark}

\newenvironment{romenumerate}[1][-10pt]{
\addtolength{\leftmargini}{#1}\begin{enumerate}
 }{\end{enumerate}}

\newcounter{thmenumerate}

\newcommand\pfitemx[1]{\par#1:}
\newcommand\pfitemref[1]{\pfitemx{\ref{#1}}}

\newcommand{\refT}[1]{Theorem~\ref{#1}}
\newcommand{\refTs}[1]{Theorems~\ref{#1}}
\newcommand{\refC}[1]{Corollary~\ref{#1}}

\newcommand{\refL}[1]{Lemma~\ref{#1}}

\newcommand{\refR}[1]{Remark~\ref{#1}}

\newcommand{\refS}[1]{Section~\ref{#1}}
\newcommand{\refSs}[1]{Sections~\ref{#1}}
\newcommand{\refSS}[1]{Section~\ref{#1}}
\newcommand{\refSSs}[1]{Sections~\ref{#1}}

\newcommand{\refE}[1]{Example~\ref{#1}}

\newcommand{\refApp}[1]{Appendix~\ref{#1}}

\newcommand{\refCond}[1]{Condition~\ref{#1}}
\newcommand{\refConds}[1]{Conditions~\ref{#1}}

\begingroup
  \divide\count255 by 60
  \multiply\count255 by -60
  \advance\count255 by \time
  \ifnum \count255 < 10 \xdef\klockan{\the\count1.0\the\count255}
  \else\xdef\klockan{\the\count1.\the\count255}\fi
\endgroup

\newcommand{\sumin}{\sum_{i=1}^n}

\newcommand{\sumkn}{\sum_{k=1}^n}
\newcommand{\sumln}{\sum_{\ell=1}^n}
\newcommand{\prodin}{\prod_{i=1}^n}

\newcommand\set[1]{\ensuremath{\{#1\}}}
\newcommand\bigset[1]{\ensuremath{\bigl\{#1\bigr\}}}
\newcommand\Bigset[1]{\ensuremath{\Bigl\{#1\Bigr\}}}

\newcommand\numset[1]{\#\set{#1}}
\newcommand\xpar[1]{(#1)}
\newcommand\bigpar[1]{\bigl(#1\bigr)}
\newcommand\Bigpar[1]{\Bigl(#1\Bigr)}

\newcommand\lrpar[1]{\left(#1\right)}
\newcommand\bigsqpar[1]{\bigl[#1\bigr]}
\newcommand\sqpar[1]{[#1]}

\newcommand\lrsqpar[1]{\left[#1\right]}
\newcommand\cpar[1]{\{#1\}}
\newcommand\bigcpar[1]{\bigl\{#1\bigr\}}

\newcommand\lrcpar[1]{\left\{#1\right\}}

\newcommand\bigabs[1]{\bigl\lvert#1\bigr\rvert}
\newcommand\Bigabs[1]{\Bigl\lvert#1\Bigr\rvert}

\def\rompar(#1){\textup(#1\textup)}    
\newcommand\xfrac[2]{#1/#2}

\newcommand\parfrac[2]{\lrpar{\frac{#1}{#2}}}

\def\xexp(#1){e^{#1}}
\newcommand\ceil[1]{\lceil#1\rceil}

\newcommand\floor[1]{\lfloor#1\rfloor}

\newcommand\setn{\set{1,\dots,n}}

\newcommand\ktoo{\ensuremath{{k\to\infty}}}
\newcommand\kktoo{\ensuremath{{\kk\to\infty}}}

\newcommand\punkt{\xperiod}    
\newcommand\iid{i.i.d\punkt}    

\newcommand\eg{e.g\punkt}

\newcommand\ii{\mathrm{i}}

\newcommand{\tend}{\longrightarrow}
\newcommand\dto{\overset{\mathrm{d}}{\tend}}

\newcommand\eqd{\overset{\mathrm{d}}{=}}

\newcommand\bbR{\mathbb R}

\newcommand\bbN{\mathbb{N}}
\newcommand\bbNo{\mathbb{N}_0}
\newcommand\bbT{\mathbb T}

\newcommand\bbZ{\mathbb Z}

\newcommand\bbB{\mathbb{B}}
\newcommand\bbD{\mathbb{D}}

\newcounter{CC}
\newcommand{\CC}{\stepcounter{CC}\CCx} 
\newcommand{\CCx}{C_{\arabic{CC}}}     
\newcommand{\CCreset}{\setcounter{CC}0} 
\newcounter{cc}
\newcommand{\cc}{\stepcounter{cc}\ccx} 
\newcommand{\ccx}{c_{\arabic{cc}}}     
\newcommand{\ccdef}[1]{\xdef#1{\ccx}}     
\newcommand{\ccname}[1]{\cc\ccdef{#1}}    
\newcommand{\ccreset}{\setcounter{cc}0} 

\renewcommand\Re{\operatorname{Re}}

\newcommand\E{\operatorname{\mathbb E{}}}
\renewcommand\P{\operatorname{\mathbb P{}}}

\newcommand\Var{\operatorname{Var}}

\newcommand\N{\mathrm{N}}
\newcommand\Po{\operatorname{Po}}
\newcommand\Bi{\operatorname{Bi}}

\newcommand\Be{\operatorname{Be}}

\newcommand\supp{\operatorname{supp}}

\newcommand\ga{\alpha}

\newcommand\gd{\delta}
\newcommand\gD{\Delta}
\newcommand\gf{\varphi}

\newcommand\kk{\kappa}
\newcommand\gl{\lambda}

\newcommand\gs{\sigma}

\newcommand\gss{\sigma^2}

\newcommand\gU{\Upsilon}
\newcommand\eps{\varepsilon}
\renewcommand\phi{\xxx}  

\newcommand\cA{\mathcal A}

\newcommand\cD{\mathcal D}

\newcommand\cI{\mathcal I}

\newcommand\cT{{\mathcal T}}

\newcommand\indic[1]{\boldsymbol1\cpar{#1}} 
\newcommand\lrindic[1]{\boldsymbol1\lrcpar{#1}} 
\newcommand\bigindic[1]{\boldsymbol1\bigcpar{#1}}

\newcommand\indicz[1]{\boldsymbol1_{\cpar{#1}}}

\newcommand\qw{^{-1}}

\newcommand\qq{^{1/2}}
\newcommand\qqw{^{-1/2}}

\newcommand\intoooo{\int_{-\infty}^\infty}
\newcommand\intpipi{\int_{-\pi}^{\pi}}

\newcommand\dtv{d_{\mathrm{TV}}}

\newcommand\dd{\,\mathrm{d}}

\newcommand{\chf}{characteristic function}

\newcommand{\ui}{uniformly integrable}

\newcommand\rhs{right-hand side}

\newcommand\GW{Galton--Watson}
\newcommand\GWt{\GW{} tree}
\newcommand\cGWt{conditioned \GW{} tree}

\newcommand\bd{\mathbf{d}}
\newcommand\hd{\widehat{d}}
\newcommand\bhd{\widehat{\bd}}
\newcommand\xk{^{(k)}}
\newcommand\bhdk{\bhd\xk}
\newcommand\hdk{\hd\xk}
\newcommand\bp{\mathbf{p}}
\newcommand\bn{{\mathbf{n}}}
\newcommand\bm{{\mathbf{m}}}
\newcommand\bmx{\mathbf{n}}
\newcommand\bnk{{\mathbf{n}_\kk}}

\newcommand\px{p}
\newcommand\gfsm{\gf_{S_m}}

\newcommand\coloneqq{:=}
\newcommand\sT{\mathsf{T}}
\newcommand\bxd{\bar d}
\newcommand\Smk{S_{m_\kk}}
\newcommand\muk{{\mu_\kk}}
\newcommand\gsk{\gs_\kk}
\newcommand\gssk{\gsk^2}
\newcommand\gsp{\gs_\bp}
\newcommand\gssp{\gsp^2}
\newcommand\hgsk{\widehat{\gs}_\kk}
\newcommand\hgssk{\hgsk^2}
\newcommand\Yk{{Y_\kk}}
\newcommand\Do{\cD_0}
\newcommand\Dp{D_{\bp}}
\newcommand\cTbn{\cT_{\bn}}
\newcommand\cTbnk{\cT_{\bnk}}
\newcommand\Dk{D_\kk}
\newcommand\xbnk{|\bnk|}
\newcommand\mk{{m_{\kk}}}

\newcommand\hgl{\widehat{\gl}}
\newcommand\bmk{{\bm_\kk}}
\newcommand\ENTT{\E[N_{T}(\cTbn)]}
\newcommand\ENTTk{\E[N_{T_\kk}(\cTbnk)]}
\newcommand\simx{\in}

\begin{document}

\begin{abstract} 

This paper extends the study of fringe trees in random plane trees with a given degree statistic. While previous work established the asymptotic normality of the count of fringe trees isomorphic to a fixed tree, we investigate the case where the target tree grows with the size of the random tree.

We consider three primary subtree counts: the number of fringe trees
isomorphic to a specific growing tree, the number of fringe trees sharing a
given growing degree statistic, and the number of fringe trees of a specific
growing size. To establish our results, we employ and compare four distinct
probabilistic frameworks: the method of moments with the Gao--Wormald theorem, Stein's method with coupling (to provide explicit error bounds in total variation distance), the Cai--Devroye method, and Stein's method with exchangeable pairs. Our findings provide conditions for Poisson and normal convergence for these subtree counts. 

Additionally, we provide a local limit theorem for sums of values obtained
via sampling without replacement that may be of independent
interest. Finally, our results and methods are also applied to conditioned
critical Galton–Watson trees. 
\end{abstract}

\maketitle

\section{Introduction}\label{S:intro}

In \cite{SJ378}, we investigated the fringe trees of random plane trees
$\cT_\bn$ with a given degree statistic $\bn$ (see \refS{SBackground} for
notation). Specifically, for a sequence of degree statistics
$(\bnk)_{\kk\ge1}$ with size $|\bnk|\to\infty$, we established, under
suitable conditions, the asymptotic normality of the number of fringe trees
in $\cT_\bnk$ equal (i.e.\ isomorphic) to a fixed tree $T$. In the present
paper, we extend this analysis to the number 
$N_{T_{\kk}}(\cT_\bnk)$
of fringe trees of
$\cT_\bnk$
that are isomorphic to some
given tree $T_{\kk}$ that grows with $\kk$. More generally, we  also
consider the number of fringe trees 
that belong to the set of trees with given degree statistic
$\bm_\kk$,
which we denote by $N_{\bm_\kk}(\cT_\bnk)$,
 and the number of fringe trees of a given size
$m_\kk$, denoted by $N_{m_\kk}(\cT_\bnk)$.

Fringe subtrees serve as a fundamental object of study in random tree
theory, starting with the fundamental work of Aldous \cite{Aldous1991} and
extensive modern probabilistic research \cite{Fuchs2012, SJ296, SJ285}; see
also references therein for further work in this area.
Beside their mathematical interest, they are also used in applications.
In computer science, they are essential to minimal Directed Acyclic Graph (DAG) compression---a widely used lossless technique that reduces a tree's size to its number of distinct fringe subtrees by merging identical roots. Beyond computer science, fringe subtrees and subtree patterns play an important role in mathematical biology, such as in the study of phylogenetic trees \cite{Chang}.

Our main results are structured as follows. 
In Section \ref{Stype}, we study $N_{T_{\kk}}(\cT_\bnk)$.
We do this using four different methods, in separate subsections,
but the main results are obtained in the first two subsections.
We use several different methods both because the results to some extent
complement each other, and because we find it interesting to show how
different methods can be applied to this problem, and to compare the results.
The first subsection derives results by explicitly
computing the factorial moments of $N_{T_{\kk}}(\cT_\bnk)$ and analysing
their asymptotic behaviour. In particular, we establish asymptotic normality
by applying the Gao-Wormald theorem \cite[Theorem 1]{GW2004} (see also
\cite[Theorem A.1]{SJ378}). 
This leads to
Theorems \ref{The01} and \ref{The1} which provide conditions for
$N_{T_{\kk}}(\cT_\bnk)$ to be asymptotically Poisson or normal. 
In \refSS{SSstein2} we use instead Stein's method with coupling \cite{SJI}
(with some calculations of first and second moments here too).
This yields a very similar result with Poisson or normal convergence
in \refC{Cstein}. This overlaps to a large extent with the result in
\refT{The1}, but the results in the two sections also complement each other;
\refT{The1} includes results on moment convergence, while Stein's method 
yields \refT{Tstein} with an explicit estimate of the error (in total
variation distance) for the approximations.
In \refSS{SSCai}, we apply a different, and simpler, 
method by Cai and Devroye \cite{Cai2017}. 
We show  that this method also gives a bound for the error in total
variation distance, which in most cases is sufficient to show convergence in
distribution to a Poisson or normal distribution; however, this bound is
weaker than the bound obtained in \refSS{SSstein2} by Stein's method.
Finally, in \refSS{SSstein1}, we discuss briefly another version of Stein's
method, with exchangeable pairs. It turns out that for the present problem, 
this gives us the same bound as the simpler method in \refSS{SSCai};
we therefore do not treat this method in detail.

In \refS{Sstatistic}, we study $N_\bmk(\cTbnk)$, the number of fringe trees
with a given degree statistics but otherwise arbitrary shape.
\refT{ThPo2} gives results on Poisson and normal convergence, corresponding
to \refT{The1}. We prove this using moment calculations as in \refSS{SSGao}.
It seems straightforward to use Stein's method as in \refSS{SSstein2}
to obtain bounds for the approximation error in total variation here too,
but we leave that to the readers.
(The methods of \refSSs{SSCai} and \ref{SSstein1} are not applicable here,
since they are based on the existence of different fringe trees with the
same degree statistics, of which some are counted and some not.)

Section \ref{Ssize} then focuses on the asymptotic behaviour of the size-specific count $N_{m_\kk}(\cT_\bnk)$. Specifically, Theorem \ref{TS} establishes that $N_{m_\kk}(\cT_\bnk)$ converges to a Poisson random variable whenever $m_\kk\sim a\xbnk^{2/3}$ as \kktoo, for some constant $a>0$. The proof uses a local limit theorem for sums of values obtained via sampling without replacement, Theorem \ref{TLLT}, which may be of independent interest.

Finally, in Section \ref{AppGW}, we apply our results to conditioned
critical Galton--Watson trees. In particular, 
Theorem \ref{TheoGW2} provides an alternative proof of results in
\cite[Theorem 1.4 (ii)]{Cai2017},
while
Theorem \ref{TheoGW1} extends
\cite[Theorem 1.3]{Cai2017} to offspring distributions with infinite
variance.

In \refApp{ALLT}, we prove a local limit theorem for the sum of a number of values
obtained by drawing without replacement that we need in \refS{Ssize}; this is presumably known, but we
have not found a reference and include a proof for completeness.

To conclude this introduction and clarify our contributions, we compare our methodology with that of \cite{SJ378}. Although we adopt the explicit expression for the factorial moments $N_{T_{\kk}}(\cT_\bnk)$ from \cite{SJ378}, the remaining arguments are largely distinct. The primary point of overlap is our use the Gao--Wormald theorem, which is similar to the approach used in \cite{SJ378} for proving the asymptotic normality of $N_{T}(\cT_\bnk)$ for a fixed tree $T$. Beyond this, our framework introduces two major methodological differences: first, the application of Stein's method to derive quantitative error bounds for the Poisson approximation; second, the use of a local limit theorem for sums of values obtained via sampling without replacement to prove the convergence of $N_{m_\kk}(\cT_\bnk)$.

\section{Some background and notation} \label{SBackground}

\subsection{Some general notation}\label{SSnot}

Let $\bbNo:=\set{0,1,\dots}$ and $\bbN:=\set{1,\dots}$. For $n \in \bbN$, let $[n] \coloneqq \{1, \dots, n\}$. 

For $x \in \mathbb{R}$ and $q \in \mathbb{N}_{0}$, let $(x)_{q} \coloneqq
x(x-1)\cdots(x-q+1)$ denote the $q$th falling factorial of $x$. 
(Here $(x)_{0} \coloneqq 1$. 
Note that $(x)_{q} = 0$ whenever $x \in \mathbb{N}_{0}$ and $x-q+1\leq0$.)

We let
$C$ and $c$ denote unspecified positive constants that may vary from one occurrence
to the next; we may add subscripts to distinguish some of them.

We use $\dto$ for convergence in distribution, 
for a sequence of random variables in some metric space. 
Also, $\stackrel{\rm d}{=}$ means equal in distribution. 
We write ${\rm N}(0,1)$ for a standard normal distribution
with mean $0$ and variance $1$. For $\lambda \in [0,\infty)$, we write 
$\Po\bigpar{\gl}$ for a Poisson distribution of parameter $\lambda$. For $m \in \mathbb{N}_{0}$ and $p \in (0,1)$, we write $\Bi\bigpar{m,p}$ for a Binomial distribution of parameters $m$ and $p$. 

Let $\bp=(p_i)_{i \in \mathbb{Z}}$ be a probability distribution on $\bbZ$. 
We define its  \emph{support} $\supp(\bp):=\set{i \in \bbZ:p_i>0}$,
and let the
\emph{span} of $\bp$ be the largest integer $h$ such that 
$i-j$ is a multiple of $h$ for all $i,j\in\supp(\bp)$;
if $\bp$ is concentrated at a single point, we define the span to be $\infty$.
In symbols, the span is
$\gcd\set{i-j:i,j\in\supp(\bp)}$.
We will usually consider $\bp$ with $p_0>0$, and then the span equals
$\gcd(\supp(\bp))$.
We say that $\bp$ is \emph{nonlattice} if its span equals 1, and
\emph{lattice} otherwise.

The total variation distance between two random variables $X$ and $Y$
in a metric space (or rather between their distributions) 
is defined by
\begin{align}
  \dtv(X,Y):=\sup_A\bigabs{\P(X\in A)-\P(Y\in A) },
\end{align}
taking the supremum over all measurable subsets $A$.

Unspecified limits are as \kktoo.

\subsection{Trees with given vertex degrees}\label{SStrees}

All trees in this paper are rooted and plane 
(also called ordered rooted trees); see e.g.\ \cite{Drmota2009}. 
Let $\mathbb{T}$ be the set of all (finite) plane rooted trees.
Denote the size, i.e.\ the
number of vertices, of a tree $T$ by $|T|$. 
Let $\bbT_n:=\set{T\in\bbT:|T|=n}$, for $n\ge1$.

The (out)degree of a vertex $v
\in T$, denoted $d_{T}(v)$, is its number of children in $T$. The \emph{degree statistic} of a rooted tree $T$ is the sequence $\mathbf{n}_{T} =
(n_{T}(i))_{i \geq 0}$, where $n_{T}(i) := \numset{v \in T: d_{T}(v) = i}$ is
the number of vertices of $T$ with $i$ children.
We have
\begin{align} \label{eq30}
|T| = \sum_{i \geq 0} n_{T}(i) = 1 + \sum_{i \geq 0} in_{T}(i).
\end{align}
Moreover, 
a sequence $\mathbf{n} = (n(i))_{i \geq 0}$ of nonnegative integers is the degree
statistic of some tree if and only if 
\begin{align}\label{ds}
\sum_{i \geq 0} n(i) = 1 + \sum_{i \geq 0} in(i).   
\end{align}
We say that $\bn$ is a \emph{degree statistic}
if \eqref{ds} holds.
In this case, we let 
$|\mathbf{n}| \coloneqq \sum_{i \geq 0} n(i)$, 
and we
write $\mathbb{T}_{\mathbf{n}}$ for
the set of plane rooted trees with degree statistic $\mathbf{n}$. It is well known that 
(see  e.g.\ \cite[Exercise 6.2.1]{Pitman})
\begin{eqnarray}\label{sw1}
|\mathbb{T}_{\mathbf{n}}| = \frac{1}{|\mathbf{n}|} \binom{|\mathbf{n}|}{\mathbf{n}} = \frac{1}{|\mathbf{n}|} \frac{|\mathbf{n}|!}{\prod_{i \geq 0} n(i)!}.
\end{eqnarray}
\noindent
We let $\mathcal{T}_{\mathbf{n}}$ be a uniformly random element of the set
$\mathbb{T}_{\mathbf{n}}$, and  we denote this by  
$\mathcal{T}_{\mathbf{n}} \simx {\rm Unif}(\mathbb{T}_{\mathbf{n}})$.

\subsection{Degree statistics and empirical degree distributions}
\label{SSdegree}

For a degree statistic $\mathbf{n}$, denote by $\mathbf{p}(\mathbf{n}) = (p_{i}(\mathbf{n}))_{i \geq 0}$ its \emph{empirical degree distribution}, i.e., 
\begin{align}\label{pin}
p_{i}(\mathbf{n}) := \frac{n(i)}{|\mathbf{n}|}, \hspace*{4mm} \text{for} \hspace*{2mm} i \geq 0. 
\end{align}
Note that this is the distribution of the degree of a uniformly
random vertex in any tree $T$ with degree statistic $\bn$.

In this paper (as in \cite{SJ378}), we assume  the following condition.
\begin{condition} \label{Condition1}
$\mathbf{n}_{\kappa} = (n_{\kappa}(i))_{i \geq 0}$, $\kk\ge1$, 
are degree
statistics such that 
as $\kappa \rightarrow \infty$:
\begin{enumerate}[label=\upshape(\roman*)]
\item\label{Cond1a} 
$|\mathbf{n}_{\kappa}| \rightarrow \infty$, 
\item \label{Cond1b}
For every $i \geq 0$, we have
$p_{i}(\mathbf{n}_{\kappa}) \rightarrow p_{i}$, 
where $\mathbf{p} := (p_{i})_{i \geq 0}$ is a 
probability distribution on $\bbNo$.
\end{enumerate} 
\end{condition}

\begin{remark}\label{RD}
Let $D_\kk$ denote a random variable with the distribution $\bp(\bnk)$;
by \eqref{pin}, $D_\kk$ can be regarded as the degree of a randomly chosen
vertex of any tree 
with degree statistic $\bnk$.
Similarly, let $\Dp$ be a random variable with distribution $\bp$.
Then \refCond{Condition1}\ref{Cond1b} says that $D_\kk\dto\Dp$, as $\kk \to \infty$.

Furthermore,
\eqref{ds} implies
\begin{align}\label{the}
\E D_\kk
=\frac{\xbnk-1}{\xbnk}
=1-\frac{1}{\xbnk},
\end{align}
and
it follows from \eqref{the} and Fatou's lemma that \refCond{Condition1} implies 
\begin{align}\label{klas}
  \sum_{i\ge0}ip_i=\E D_\bp\le1.
\end{align}
We have equality in \eqref{klas} if and only if the variables $D_\kk$ are
uniformly integrable
(see e.g.\ \cite[Theorem 5.5.9]{Gut}).
\end{remark}

We note also that,
similarly to the asymptotic result \eqref{klas},
for any degree statistic $\bn$, we have by \eqref{ds}
\begin{align}\label{ari1}
  \sum_{i\ge0}ip_i(\bn) = 1-\frac{1}{|\bn|} < 1.
\end{align}
In particular, it follows that
\begin{align}\label{ari2}
  \sum_{i\ge2}p_i(\bn)  \le \tfrac12  \sum_{i\ge2}ip_i(\bn) <\tfrac12,
\end{align}
and thus 
\begin{align}\label{ari3}
p_0(\bn)+p_1(\bn)>\tfrac12.   
\end{align}

In \refS{Ssize}, but \emph{not} in \refSs{Stype}--\ref{Sstatistic},
we will also need the following condition.

\begin{condition}\label{Cond2}
  In addition to \refCond{Condition1},
we have $\sum_{i\ge0} i^2 p_i(\bnk) \to \sum_{i\ge0} i^2p_i<\infty$, as $\kk \to \infty$.
\end{condition}

\begin{remark}\label{RD2}
Assume \refCond{Condition1} and
let $D_\kk$ and $\Dp$ be as in \refR{RD}.
Then \refCond{Cond2} says that $\E D_\kk^2\to\E\Dp^2<\infty$ as $\kk \to
\infty$, 
which is equivalent to
the sequence $D_\kk^2$ being uniformly integrable. In particular, this implies $\E D_\kk\to\E\Dp$ as $\kk \to \infty$, and thus
$\E\Dp=1$ by  
\eqref{the}.
Similarly, it is easy to see that \refCond{Cond2} 
then is equivalent to
\begin{align}\label{cond2b}
\gssk:=  \Var D_\kk \to \gssp:=\Var \Dp = \sum_{i\ge0}(i-1)^2p_i, \quad \text{as} \quad  \kk \to \infty.
\end{align}
\refCond{Cond2} is thus a natural analogue of 
assuming finite offspring variance for
a conditioned Galton--Watson tree.
\end{remark}

\subsection{Fringe trees}\label{SSfringe}

For $T \in \mathbb{T}$ and a vertex $v \in T$, let $T_{v}$ be the subtree of
$T$ rooted at $v$, i.e., the subtree consisting of $v$ and all its
descendants. We call $T_{v}$ a fringe (sub)tree of $T$.
We regard $T_v$ as an element of $\bbT$ and
let, for $T, T^{\prime} \in \mathbb{T}$, 
\begin{align} \label{eq2}
N_{T^{\prime}}(T) \coloneqq  \numset{v \in T: T_{v} = T^{\prime}} 
= \sum_{v \in T} \indicz{T_{v} = T^{\prime}},
\end{align}
i.e., the number of fringe subtrees of $T$ that are equal 
to $T^{\prime}$. 
(Recall that $T_v=T'$ in $\bbT$ means that $T_v$ is isomorphic to
$T'$ as a plane rooted tree.) For a degree statistic $\bn$,  we let
\begin{align} \label{eq2s}
N_{\bn}(T) \coloneqq  \numset{v \in T: T_{v} \in \bbT_\bn} 
=\sum_{T'\in\bbT_{\bn}}N_{T'}(T)
\end{align}
be the number of fringe trees of $T$ with degree statistic $\bn$. Furthermore, we let, for an integer $m\ge1$,
\begin{align} \label{eq2m}
N_{m}(T) \coloneqq  \numset{v \in T: |T_{v}| = m} 
= \sum_{v \in T} \indicz{|T_{v}| =m}
=\sum_{T' \in \mathbb{T}_{m}}N_{T'}(T),
\end{align}
the number of fringe trees of $T$ that have size $m$.

\subsection{Trees and degree sequences} \label{TreesandDegrees}
We recall some well known facts about representations of plane trees by
degree sequences, see \eg{}
\cite[\S6.1--6.2, in particular Lemma 6.3]{Pitman} and
\cite[\S15]{SJ264}.

Given a plane tree $T\in\bbT_n$, we define its 
(depth-first) \emph{degree sequence} as 
$\bd_T=\bigpar{d_T(v_1),\dots,d_T(v_n)}$,
taking the vertices of $T$ in depth-first order.
The map $\gU:T\mapsto \bd_T$ is a bijection $\bbT_n\leftrightarrow\bbD_n$,
where
\begin{align}\label{Dn}
  \bbD_n:=\Bigset{(d_1,\dots,d_n)\in\bbNo^n:
\sum_{i=1}^k d_i \ge k \text{ for } 1\le k<n, \text{ and }
\sum_{i=1}^n d_i = n-1
}.
\end{align}
Furthermore, let 
\begin{align}\label{Bn}
  \bbB_n:=\Bigset{(d_1,\dots,d_n)\in\bbNo^n:
\sum_{i=1}^n d_i = n-1}.
\end{align}
Then $\bbD_n\subseteq\bbB_n$, and given any $\bd=(d_1,\dots,d_n)\in\bbB_n$,
  exactly one of the cyclic shifts of $\bd$ belongs to $\bbD_n$.
This defines an $n$-to-1 map $\Phi:\bbB_n\to\bbD_n$, and thus an $n$-to-1 map
$\Psi:=\gU\qw\Phi:\bbB_n\to\bbT_n$.
For convenience, we regard sequences in $\bbB_n$ as cyclic, and we define
$d_i$ for all $i\in\bbZ$ by requiring that $d_{i}=d_{i-n}$ for all $i\in\bbZ$.

\begin{remark}
  These results are often stated in terms of the partial sums
  $\sum_{i=1}^k(d_i-1)$ 
and the walks (excursions, bridges) that they form. This is often convenient or important, but for our purposes it is simpler to
use the degrees $d_i$ directly.
\end{remark}

We define also, for a given degree statistic $\bn=(n(i))_{i\ge0}$, 
\begin{align}\label{Bbn}
\mathbb{B}_{\mathbf{n}} := 
\bigset{(d_1,\dots,d_{|\bn|})\in\bbNo^{|\bn|}
: \numset{j : d_j = i} = n(i), \text{ for every }i \geq 0 },
\end{align}
and
\begin{align} \label{eq2g}
\bbD_\bmx \coloneqq \gU(\bbT_\bmx)
=\bbD_{|\bmx|}\cap\bbB_\bmx,
\end{align}
the set of all (depth-first) degree sequences of  trees in $\bbT_\bmx$.
Note that \eqref{ds} implies $\bbB_\bn\subseteq\bbB_{|\bn|}$.
Furthermore, $\bbB_\bn=\Psi\qw(\bbT_\bn)$.
Consequently, for a given degree statistic $\bn$, 
we may construct a uniformly random tree $\cT_\bn\in\bbT_\bn$ by
\begin{align}\label{ctd}
  \cT_\bn:=\Psi(\bd),
\end{align}
for a uniformly random sequence $\bd\in\bbB_\bn$.
Note that we may here in turn construct 
$\bd$ by
\begin{align}\label{ctd2}
\bd:=(d'_{\tau(1)},\dots,d'_{\tau(|\bn|)})  
\end{align}
for any fixed sequence $(d_1',\dots,d_{|\bn|}')\in\bbB_\bn$ (for example, $n(0)$
0's followed by $n(1)$ 1's, and so on), and a uniformly random permutation
$\tau$ of $\{1, \dots, |\bn|\}$.
We assume these constructions in the sequel.

Moreover, see \eg{} \cite[\S17]{SJ264}, if $T$ has degree sequence
$d_{T}(v_1),\dots,d_{T}(v_{|T|})$, then the fringe tree at $v_j$ has degree sequence
$d_{T}(v_j),\dots,d_{T}(v_k)$ for the unique $k\ge j$ such that this is a degree
sequence, i.e., belongs to $\bbD_{k-j+1}$.
(We use here that we define our degree sequences to be in
depth-first order.)
Applying this to $\cT_\bn$,
it follows that for any given plane tree $T$,
if we denote the degree sequences of the trees by 
$\bd_{\cTbn}=(d_{\cTbn,1},\dots,d_{\cTbn,|\bn|})$ and
$\bd_T=(d_{T,1},\dots,d_{T,|T|})$, 
then
\begin{align}\label{ari4}
  N_T(\cT_\bn) = 
\sum_{k=1}^{|\bn|-|T|+1}\bigindic{(d_{\cTbn,k+i-1})_{i=1}^{|T|}=\bd_{T}}.
\end{align}
We regard the degree sequence $\bd_{\cTbn}$ as cyclic, with indices
interpreted modulo $|\bn|$. Then the sum in \eqref{ari4} remains the same if we
extend  the sum to $\sum_{k=1}^{|\bn|}$, since the indicators for $k>|\bn|-|T|+1$
vanish. 
(A sequence $(d_{\cTbn,k+i-1})_{i=1}^{|T|}$ that includes the index $|\bn|$ 
and then continues, i.e.\ wraps around,
can never be the degree sequence of a tree.)
Furthermore, the latter sum is invariant under cyclic shifts of 
the degree sequence $\bd_{\cTbn}$. 
Hence, constructing $\cT_\bn$ as in \eqref{ctd} for a
uniformly random sequence  $\bd=(d_1,\dots,d_{|\bn|})$ in $\bbB_\bn$,
we have
\begin{align}\label{b1}
  N_T(\cT_\bn)=\sum_{j=1}^{|\bn|}\bigindic{(d_{j+i-1})_{i=1}^{|T|}=\bd_{T}},
\end{align}
where 
we recall that $d_\ell$ is interpreted cyclically.

\subsection{\GWt{s}}

For a probability distribution $\mathbf{p} = (p_{i})_{i \geq  0}$ on $\mathbb{N}_{0}$, let $\mathcal{T}_{\mathbf{p}}$ be a
Galton--Watson tree with offspring distribution $\bp$, and define
$\pi_{\mathbf{p}}$ as the distribution of  $\mathcal{T}_{\mathbf{p}}$, i.e.,
(with $0^0:=1$ as usual)
\begin{align}\label{pip}
\pi_{\mathbf{p}}(T) 
\coloneqq 
\P\xpar{\mathcal{T}_{\mathbf{p}}=T}
=\prod_{i \geq 0} p_{i}^{n_{T}(i)}
=\prod_{i \in\cD(T)} p_{i}^{n_{T}(i)}
, \qquad \text{for} \, \, \, T \in \mathbb{T},
\end{align}
where 
\begin{align}\label{cD}
\cD(T):=\set{i \geq 0:n_T(i)>0}=\set{d_T(v):v\in T},   
\end{align}
the set of degrees that appear in $T$. 
 We note the simple estimate
\begin{align}  \label{ny2}
0\le \pi_{\mathbf{p}}(T) 
\le p_i\bigpar{\max_{j \in\cD(T)} p_{j}}^{|T|-1}
\le \bigpar{\max_{j \in\cD(T)} p_{j}}^{|T|},
\qquad i\in \cD(T)
.\end{align}

\section{Fringe trees of a given type (depending on $\kk$)}
\label{Stype}

We will in this section study the number of fringe trees of $\cT_{\bn_\kk}$
equal (isomorphic) to some given tree, for a given 
sequence of degree statistics $(\bn_\kk)_{\kk\ge1}$.
We studied in \cite{SJ378} the case of a fixed tree $T$;
here we consider the general case of a tree $T_\kk$ that depends on the
degree statistic $\bn_\kk$, with focus on the case $|T_\kk|\to\infty$.
Recalling the notation in \refS{S:intro}, we thus study $N_{T_\kk}(\cT_{\bn_\kk})$
for given sequences $T_\kk$ and $\bn_\kk$, $\kk\ge1$.

\subsection{Moments and the method of Gao and Wormald}
\label{SSGao}

We begin with a formula (from \cite{SJ378})
and some estimates for the expectation.
Recall the notation \eqref{pin} and \eqref{pip}--\eqref{cD}.
(Similar formulas for higher moments exist too,
see \cite{SJ378} and \eqref{eq4G} below.) 

\begin{lemma}[Partly \cite{SJ378}]\label{LE1}
  Let $\bn$ be a degree statistic and let $T\in\bbT$ with $|T|\le|\bn|$.
Then
\begin{align}\label{le1a}
  \E N_T(\cT_\bn) 
= \frac{|\bn|}{(|\bn|)_{|T|}} \prod_{i\ge0}(n(i))_{n_T(i)}
= \frac{|\bn|}{(|\bn|)_{|T|}} \prod_{i\in\cD(T)}(n(i))_{n_T(i)}
.\end{align}
Hence,
\begin{align}\label{le1b}
  \E N_T(\cT_\bn) 
\le \frac{|\bn|^{|T|+1}}{(|\bn|)_{|T|}} \prod_{i\in\cD(T)}p_i(\bn)^{n_T(i)}
=\frac{|\bn|^{|T|+1}}{(|\bn|)_{|T|}} \pi_{\bp(\bn)}(T)
,\end{align}
and for every $C>0$ there is a $C'\geq 0$ such that if\/ $|T|^2\le C|\bn|$, then
\begin{align}\label{le1c}
  \E N_T(\cT_\bn) 
\le |\bn|\pi_{\bp(\bn)}(T)\cdot\Bigpar{1+C'\frac{|T|^2}{|\bn|}}.
\end{align}
\end{lemma}
\begin{proof}
The formula \eqref{le1a} is \cite[Lemma 3.1]{SJ378}.  

The estimate \eqref{le1b} follows from \eqref{le1a} and
$(n(i))_{n_T(i)}\le n(i)^{n_T(i)}$, recalling \eqref{pin}, \eqref{eq30}, and
\eqref{pip}. 

Finally, \eqref{le1c} follows from \eqref{le1b} and the
useful estimate (see for example \cite[Lemma 4.1]{SJ378}), 
\begin{align} \label{eq5Gb}
(x)_{k} = x^{k} \exp \left( O\left(\frac{k^2}{x} \right) \right),
\end{align}
for $x \geq 1$ a real number and $0 \leq k \leq x/2$ an integer.
\end{proof}

\refL{LE1} shows loosely that $|\bn|\pi_{\bp(\bn)}$ can be regarded as an
approximation of $\E N_T(\cT_\bn)$. This is seen more formally in the
theorems below, and it also motivates our use of 
$|\bn|\pi_{\bp(\bn)}$ in the statements.

We state in \refT{The01} below a very general theorem, which essentially
(ignoring some technicalities) shows  that 
$N_{T_\kk}(\cT_{\bn_\kk})$ is asymptotically Poisson distributed when its
mean converges to a finite value, and asymptotically normal when its mean
tends to infinity.

We will prove \refT{The01} as a corollary of the even more general
\refT{The1}, which has weaker but more technical conditions.
We will actually give two proofs of the convergence in distribution in
\refT{The1} (and thus \refT{The01}), 
one below using the method of moments and the Gao--Wormald theorem, and the
other in \refS{SSstein2} using Stein's method. We find it interesting that
these quite different methods (at least in our versions)
both seem to require almost the same conditions.
The two methods also give different extra results as a bonus: the method
of moments shows moment convergence, and Stein's method yields an explicit
estimate of the total variation distance to a Poisson distribution.

\begin{theorem} \label{The01}
Assume \refCond{Condition1} and assume that the limit distribution $\bp$
is not concentrated at $0$ or at $1$, i.e., $p_0,p_1<1$.
For $\kk \geq 1$, let $\mathcal{T}_{\mathbf{n}_{\kk}}\simx {\rm Unif}(\mathbb{T}_{\mathbf{n}_{\kk}})$ 
and let $T_{\kappa} \in \mathbb{T}$ be such that
$m_{\kappa} \coloneqq |T_{\kappa}| \rightarrow \infty$
as \kktoo.
Then, as $\kk \rightarrow \infty$:
\begin{romenumerate}[-20pt]  
\item \label{ReGD10} 
If\/ $|\bn_{\kk}| \pi_{\mathbf{p}(\bn_{\kk})}(T_{\kappa}) \rightarrow \lambda$, 
for some $\lambda \in [0,\infty)$, then 
\begin{align} \label{regd10}
N_{T_\kk}(\cTbnk) \dto\Po\bigpar{\gl}
\end{align} 
with convergence of all moments.

\item \label{ReGD20} 
If\/ $|\bn_{\kk}| \pi_{\mathbf{p}(\bn_{\kk})}(T_{\kappa}) \rightarrow \infty$, then 
\begin{align}\label{regd20}
\frac{N_{T_\kk}(\cTbnk)- |\bn_{\kk}| \pi_{\mathbf{p}(\bn_{\kk})}(T_{\kappa})}{\sqrt{|\bn_{\kk}| \pi_{\mathbf{p}(\bn_{\kk})}(T_{\kappa})}} \dto\N\bigpar{0,1},
\end{align}
\end{romenumerate}  
with convergence of mean and variance.

Thus, in both cases,
\begin{align}\label{regd30}
\E\bigsqpar{N_{T_\kk}(\cTbnk)}
\sim
\Var\bigsqpar{N_{T_\kk}(\cTbnk)}
\sim
|\bn_{\kk}| \pi_{\mathbf{p}(\bn_{\kk})}(T_{\kappa}).   
\end{align}
\end{theorem}

As said above, we prove \refT{The01} as a corollary of  a more general theorem,
stated below.
Note that we assume $\gl>0$ in \refT{The1}\ref{ReGD1}, unlike in \refT{The01};
we conjecture
that the result holds also for $\gl=0$, but we leave that as an open problem.
(The case $m_\kk=O(\sqrt{|\bn_\kk|})$ is clear by \eqref{le1c}.)

\begin{theorem} \label{The1}
  Assume \refCond{Condition1}. 
For $\kk \geq 1$, let $\mathcal{T}_{\mathbf{n}_{\kk}}\simx {\rm
  Unif}(\mathbb{T}_{\mathbf{n}_{\kk}})$ 
and let $T_{\kappa} \in \mathbb{T}$ be such that
as \kktoo,
with $m_{\kappa} \coloneqq |T_{\kappa}|$,
\begin{align}
\label{rny2}
&m_\kk^2\pi_{\bp(\bn_{\kk})}(T_{\kk})\sum_{i\in\cD(T_\kk)}
\frac{p_i(\bn_{T_\kk})^2}{p_i(\bn_\kk)}
=o(1).
\end{align}
Then:
\begin{romenumerate}[-20pt]  
\item \label{ReGD1} 
If\/ $|\bn_{\kk}| \pi_{\mathbf{p}(\bn_{\kk})}(T_{\kappa}) \rightarrow \lambda$, 
for some $\lambda \in (0,\infty)$, then 
\begin{align} \label{regd1}
N_{T_\kk}(\cTbnk) \dto\Po\bigpar{\gl}
\end{align} 
with convergence of all moments.
\item \label{ReGD2} 
If\/ $|\bn_{\kk}| \pi_{\mathbf{p}(\bn_{\kk})}(T_{\kappa}) \rightarrow \infty$, then 
\begin{align}\label{regd2}
\frac{N_{T_\kk}(\cTbnk)- |\bn_{\kk}| \pi_{\mathbf{p}(\bn_{\kk})}(T_{\kappa})}{\sqrt{|\bn_{\kk}| \pi_{\mathbf{p}(\bn_{\kk})}(T_{\kappa})}} \dto\N\bigpar{0,1},
\end{align}
\end{romenumerate}  
with convergence of mean and variance.

Thus, in both cases,
\begin{align}\label{regd3}
\E\bigsqpar{N_{T_\kk}(\cTbnk)}
\sim
\Var\bigsqpar{N_{T_\kk}(\cTbnk)}
\sim
|\bn_{\kk}| \pi_{\mathbf{p}(\bn_{\kk})}(T_{\kappa}).   
\end{align}
\end{theorem}

Before proceeding to the proofs, note that \eqref{ari3} implies that, for any
tree $T$ and any degree statistic $\bn$,
\begin{align}\label{rny0}
  \sum_{i\in\cD(T)}\frac{p_i(\bn_T)^2}{p_i(\bn)}
\ge p_0(\bn_T)^2 + p_1(\bn_T)^2 \ge\frac12\bigpar{p_0(\bn_T)+p_1(\bn_T)}^2 
>\frac{1}{8}.
\end{align}
Hence, \eqref{rny2} implies 
\begin{align}\label{rny1}
 m_{\kk}^{2}  \pi_{\mathbf{p}(\bn_{\kk})}(T_{\kappa})=o(1).
\end{align}
It is easy to construct examples showing that the converse does not hold in
general; nevertheless, in typical applications
$\sum_{i\in\cD(T_\kk)}\frac{p_i(\bn_{T\kk})^2}{p_i(\bnk)}=O(1)$,
and then \eqref{rny2} is equivalent to \eqref{rny1}.

\begin{proof}[Proof of \refT{The01} from \refT{The1}]
  Since $p_0,p_1<1$, and \eqref{klas} implies $p_i\le 1/i$ for $i\ge2$, 
we have 
\begin{align} \label{EqEx1}
\max_{i \geq 0} p_{i} <1.
\end{align}
Note that \refCond{Condition1} \ref{Cond1b} says that
$\mathbf{p}(\mathbf{n}_{\kk})$ converges weakly to $\mathbf{p}$, as $\kk
\rightarrow \infty$. As is well known, this is equivalent to convergence in
total variation, and thus also to
\begin{align} \label{EqEx2}
\lim_{\kk \rightarrow \infty }\sup_{i \geq 0}|p_{i}(\mathbf{n}_{\kk}) - p_{i}|= 0.
\end{align}
Consequently, if we choose $\gd>0$ with $\max_{i\ge0}p_i<1-\gd$,
then for all sufficiently large $\kk$, we have
\begin{align} \label{EqEx3}
\max_{i \geq 0} p_{i}(\bn_\kk) \le1-\gd.
\end{align}
We consider such $\kk$ only, and then, 
using \eqref{ny2} and  $m_\kk\to\infty$,
\begin{align}\label{rny02}
&m_\kk^2\sum_{i\in\cD(T_\kk)}\frac{\pi_{\bp(\bn_{\kk})}(T_{\kk})}{p_i(\bn_\kk)}
p_i(\bn_{T_\kk})^2
\le
m_\kk^2\sum_{i\in\cD(T_\kk)}\bigpar{\max_{j\in\cD(T_\kk)}p_j(\bn_\kk)}^{m_\kk-1}
p_i(\bn_{T_\kk})^2
\notag\\&\hskip4em
\le m_\kk^2(1-\gd)^{m_\kk-1} \sum_{i\in\cD(T_\kk)}p_i(\bn_{T_\kk})
= m_\kk^2(1-\gd)^{m_\kk-1}
=o(1).
\end{align}
Thus 
\eqref{rny2} holds, and the result follows from
\refT{The1}, except in the case $\gl=0$ of Case \ref{ReGD10}.

In this exceptional case $\gl=0$, we note first that if
$m_\kk\le \sqrt{|\bn_\kk|}$, then 
\eqref{le1c} yields 
\begin{align}
\E N_{T_\kk}(\cT_{\bn_\kk}) \le C |\bn_\kk|\pi_{\bp(\bn_\kk)}(T_\kk) \to0
\end{align}
and thus \eqref{regd10} holds. We similarly obtain convergence of higher moments
$\E [N_{T_\kk}(\cT_{\bn_\kk})^q]
 \le C \xpar{|\bn_\kk|\pi_{\bp(\bn_\kk)}(T_\kk)}^q \to0$
for every fixed $q$ by \cite[Lemma 3.3]{SJ378} and \eqref{eq5Gb}. 

If $\gl=0$ and $m_\kk> \sqrt{|\bn_\kk|}$, let $m_\kk(i):=n_{T_\kk}(i)$, 
so that $\sum_{i\ge0}m_\kk(i)=m_\kk$. We may assume $m_\kk(i)\le n_\kk(i)$
since otherwise $\E N_{T_\kk}(\cT_{\bn_\kk}) =0$.
Consider the fraction
\begin{align}\label{ste1}
\frac{1}{(|\bn_\kk|)_{\sum_i m(i)}}  \prod_{i\ge0} (n_\kk(i))_{m(i)}
\end{align}
for arbitrary integer sequences $m(i)$ with $0\le m(i)\le n_\kk(i)$
(not necessarily degree statistics of any tree).
Let $m:=\sum_{i\ge0}m(i)$.
If $m(i)>0$ and we reduce $m(i)$ by 1 (keeping all other $m(j)$), then 
the fraction \eqref{ste1} is increased by a factor $\frac{|\bn_\kk|-m+1}{n_\kk(i)-m(i)+1}$,
which is $\ge1$ since
\begin{align}\label{ste2}
\bigpar{|\bn_\kk|-m+1}-\bigpar{n_\kk(i)-m(i)+1}
=\sum_{j\neq i} \bigpar{n_\kk(j)-m(j)}\ge0.  
\end{align}
We may repeat this until we have $m=\ceil{\sqrt{|\bn_\kk|}}$, and then we
obtain
using \eqref{eq5Gb} 
\begin{align}\label{ste3}
\frac{1}{(|\bn_\kk|)_{m}}  \prod_{i\ge0} (n_\kk(i))_{m(i)}  
\le \frac{C}{|\bn_\kk|^m} \prod_{i\ge0}(|\bn_\kk| p_i(\bn_\kk))^{m(i)}\le C(1-\gd)^{m}
\le C(1-\gd)^{\sqrt{|\bn_\kk|}}  
.\end{align}
Hence, the fraction \eqref{ste1} is $\le C(1-\gd)^{\sqrt{|\bn_\kk|}}$
when $\sum_{i\ge0}m(i) >\sqrt{|\bn_\kk|}$. Returning to
$m_\kk(i)=n_{T_\kk}(i)$,
we thus obtain by \eqref{le1a}
\begin{align}\label{ste4}
\E N_{T_\kk}(\cT_{\bn_\kk})  \le 
C|\bn_\kk|(1-\gd)^{\sqrt{|\bn_\kk|}}
\to0,
\end{align}
which implies \eqref{regd10}.
Again, we similarly obtain convergence also of higher moments to 0
using \cite[Lemma 3.3]{SJ378} and \eqref{eq5Gb};
we omit the details.
\end{proof}

\begin{proof}[Proof of \refT{The1}]
Let 
\begin{align}
  \label{glkk}
\gl_\kk:=|\bn_{\kk}|\pi_{\mathbf{p}(\bn_{\kk})}(T_{\kappa}).
\end{align}
Thus $\gl_\kk\to\gl\in(0,\infty]$, where we define $\gl:=\infty$ in Case
\ref{ReGD2}.
Note that, as said above, \eqref{rny2} and \eqref{rny0} imply \eqref{rny1}.

Let
throughout the proof 
$q_{\kk} \in \mathbb{N}$ be such that 
\begin{align}  \label{ny3}
q_{\kk} = O\bigpar{\gl_\kk\qq}.
\end{align}
Then \eqref{ny3}, \eqref{glkk}, \eqref{rny1}, and \eqref{rny2} imply that,
for some $C<\infty$, 
\begin{align} \label{N1e}
\frac{q^{2}_{\kk}m_{\kk}^{2}}{|\mathbf{n}_{\kk}|} &
\le C
m_{\kk}^{2}  \pi_{\mathbf{p}(\bn_{\kk})}(T_{\kappa})
= o(1),
\\\label{N1f}
\sum_{i\in \cD(T_{\kappa})} \frac{q_{\kk}^{2}n_{T_{\kappa}}(i)^{2}}{n_{\kappa}(i)}
&\le C m_\kk^2\pi_{\mathbf{p}(\bn_{\kk})}(T_{\kappa})
\sum_{i\in \cD(T_{\kappa})} \frac{p_i(\bn_{T_{\kappa}})^{2}}{p_i(\bn_{\kk})}
=o(1)
.\end{align}
In particular, 
by \eqref{N1e}, 
we have for $\kappa$ large enough 
$q_{\kappa} m_{\kappa}\le|\bn_{\kappa}|\qq\le |\bn_{\kappa}|$. Then, by \cite[Lemma 3.3 (i)]{SJ378}, 
\begin{align} \label{eq4G}
\E (N_{T_{\kappa}}(\mathcal{T}_{\mathbf{n}_{\kappa}}))_{q_{\kappa}} =
\frac{|\mathbf{n}_{\kappa}|}{(|\mathbf{n}_{\kappa}|)_{q_{\kappa}m_{\kappa}-q_{\kappa}+1}} 
  \prod_{i\in\cD(T_\kk)} (n_{\kappa}(i))_{q_{\kappa}n_{T_{_{\kappa}}}(i)}.
\end{align}
By \eqref{N1e} and \eqref{N1f}, for $\kk$ large enough,
$q_\kk m_\kk\le |\bn_\kk|/2$
and
$q_\kk n_{T_\kk}(i) \le n_\kk(i)/2$ for all $i\in\cD(T_\kk)$;
then, \eqref{eq4G}, \eqref{eq5Gb}, \eqref{pin},
\eqref{N1e}, \eqref{N1f}, \eqref{pip}, and \eqref{glkk} imply that
\begin{align}
\E (N_{T_{\kappa}}(\mathcal{T}_{\mathbf{n}_{\kappa}}))_{q_{\kk}}&   
= |\bn_{\kappa}|^{-q_\kk m_\kk+q_\kk}\prod_{i\in\cD(T_\kk)} n_\kk(i)^{q_\kk n_{T_\kk}(i)}
\cdot \exp \Bigpar{
O\Bigpar{\frac{q_{\kk}^{2}m_{\kk}^{2}}{|\bn_{\kk}| } + 
\sum_{i\in \cD(T_{\kappa})} \frac{q_{\kk}^{2}n_{T_{\kappa}}(i)^{2}}{n_{\kappa}(i)}} }
\notag\\&
= |\bn_{\kappa}|^{q_\kk}\prod_{i\in\cD(T_\kk)} p_i(\bn_\kk)^{q_\kk n_{T_\kk}(i)}
\cdot \exp \bigpar{o(1)}
\notag\\&
= \gl_\kk^{q_{\kk}}  
\cdot  \exp \bigpar{o(1)}
\label{eq11G}
.\end{align}

In Case \ref{ReGD1}, 
\eqref{ny3} holds for all $q_{\kk} = q \in \mathbb{N}$ fixed (since we have assumed $\gl>0$), 
and thus \eqref{eq11G} shows that
\begin{align}\label{eq12s} 
\E (N_{T_{\kappa}}(\mathcal{T}_{\mathbf{n}_{\kappa}}))_{q} \sim 
\gl_\kk^{q}
\to \gl^q
,\end{align}
\noindent as $\kk \rightarrow \infty$.
Hence  the claim in \ref{ReGD1} follows by the methods of moments;
\eqref{regd3} follows too.

In Case \ref{ReGD2}, it follows from \eqref{eq11G} that, 
for $q_{\kk}= O(\gl_\kk\qq)$,
\begin{align} 
\E (N_{T_{\kappa}}(\mathcal{T}_{\mathbf{n}_{\kappa}}))_{q_{\kk}}&   
= \gl_\kk^{q_{\kk}} \cdot \exp \Big(\frac{\gl_\kk-\gl_\kk}{2\gl_\kk^2}q_\kk^2
+ o(1)\Big).
\end{align}
\noindent 
Hence,  \eqref{regd2} follows by applying the 
Gao--Wormald theorem \cite[Theorem 1]{GW2004} (or \cite[Theorem A.1]{SJ378} 
with $m=1$) with 
$\mu_{\kk} = \sigma_{\kk}^{2} = \gl_\kk$.
Moreover, for a fixed $q$, \eqref{N1e} and \eqref{N1f} are improved to
\begin{align} \label{N1e2}
\frac{q^{2}m_{\kk}^{2}}{|\mathbf{n}_{\kk}|} &
=q^2
\frac{m_{\kk}^{2}  \pi_{\mathbf{p}(\bn_{\kk})}(T_{\kappa})}{\gl_\kk}
=o\Bigpar{\frac{1}{\gl_\kk}},
\\\label{N1f2}
\sum_{i\in \cD(T_{\kappa})} \frac{q^{2}n_{T_{\kappa}}(i)^{2}}{n_{\kappa}(i)}
&=q^2\frac{ \pi_{\mathbf{p}(\bn_{\kk})}(T_{\kappa})}{\gl_k}
\sum_{i\in \cD(T_{\kappa})} \frac{m_\kk^2 p_i(\bn_{T_{\kappa}})^{2}}{p_i(\bn_{\kk})}
=o\Bigpar{\frac{1}{\gl_\kk}}
.\end{align}
Hence we can for fixed $q$ improve \eqref{eq11G} to
\begin{align}
\E (N_{T_{\kappa}}(\mathcal{T}_{\mathbf{n}_{\kappa}}))_{q}&   
= \gl_\kk^{q}  
\cdot  \exp \bigpar{o(1/\gl_\kk)}
= \gl_\kk^{q}  + o\bigpar{\gl_\kk^{q-1}}
\label{eq11G+}
.\end{align}
Taking $q=1$ in \eqref{eq11G+} yields
$\E\sqpar{N_{T_\kk}(\cTbnk)}=\gl_\kk+o(1)\sim \gl_\kk$, and 
then taking $q=2$ yields
\begin{align}
  \Var\bigsqpar{N_{T_\kk}(\cTbnk)}
=
\E\sqpar{(N_{T_\kk}(\cTbnk))_2}
+
\E\sqpar{N_{T_\kk}(\cTbnk)}
- \E\sqpar{(N_{T_\kk}(\cTbnk))}^2
=\gl_\kk+o(\gl_\kk),
\end{align}
which proves \eqref{regd3} and thus convergence of mean and variance in
\eqref{regd2}.
\end{proof}

\begin{remark}\label{Rbad?}
The reason that we exclude the case $\gl=0$ from \refT{The1} 
is that we have been unable to exclude the possibility that in some extreme
cases we might have 
$|\bn_\kk|\pi_{\mathbf{p}(\bn_{\kk})}(T_{\kappa}) \rightarrow 0$
and \eqref{rny2} but $\E N_{T_\kk}(\cTbnk)\not\to0$.
This cannot happen in nice cases such as \refT{The01}, and we conjecture
that it never happens, but we leave it as an open problem.
See also \refC{Cstein}, where 
$|\bn_\kk|\pi_{\mathbf{p}(\bn_{\kk})}(T_{\kappa})$
is replaced by
$\ENTTk$, and $\gl=0$ is included.
\end{remark}

\begin{remark}
  We do not know whether convergence of higher moments hold in Case
  \ref{ReGD20} of \refTs{The01} and \ref{The1}
without further assumptions, and leave that as an open problem.
\end{remark}

\begin{remark}\label{Rny}
It seems likely that
also in some cases where \eqref{rny2} does not hold, it might be
possible to obtain a central limit theorem, as in \eqref{regd2} but with a
different asymptotic variance, by using a more precise version of
\eqref{eq5Gb} with an explicit first term in the exponent, 
see \cite[Lemma 4.1]{SJ378}.
We have not explored this.
\end{remark}

\begin{example}\label{EX1}
We give an example to show that in extreme cases, when the conditions of the
theorems are not satisfied, the conclusion may fail.

For $\kk\ge2$, let $T_\kk$ be the star with $\kk$ vertices and let
  $\bn_\kk:=\bn_{T_\kk}$; thus $n_\kk(0)=\kk-1$ and $n_\kk(\kk-1)=1$.
Since $T_\kk$ is the only tree with this degree statistic, we have
$\cT_{\bn_\kk}=T_\kk$ a.s., and thus 
\begin{align}\label{ex1a}
N_{T_\kk}(\cT_{\bn_\kk})=1.  
\end{align}
On the other hand, we have $p_0(\bn_\kk)=(\kk-1)/\kk=1-1/\kk$ and
$p_{\kk-1}(\bn_\kk)=1/\kk$, and thus 
\begin{align}\label{ex1b}
|\bn_\kk|\pi_{\bp(\bn_\kk)}(T_\kk) =\kk\Bigpar{1-\frac{1}{\kk}}^{\kk-1}\frac{1}{\kk}  
=\Bigpar{1-\frac{1}{\kk}}^{\kk-1}
\to e\qw.
\end{align}
Hence, in this example, $|\bn_\kk|\pi_{\bp(\bn_\kk)}(T_\kk)$ converges, but
the Poisson convergence \eqref{regd1} fails (for any $\gl$); moreover,
$|\bn_\kk|\pi_{\bp(\bn_\kk)}(T_\kk)\sim e\qw$ 
is a rather bad approximation of $\E N_{T_\kk}(\cT_{\bn_\kk})=1$.
It is easily verified that the condition \eqref{rny2} does not hold in this
example. 
\end{example}

\begin{example}
We give an example to show that the Poisson case of \refT{The01} may occur. Let $\mathbf{n}_{\kappa} = (n_{\kappa}(i))_{i \geq 0}$, $\kk \geq 1$, be
degree statistics such that \refCond{Condition1} is satisfied. Let $\mathcal{I}$ be a finite subset of $\{2, 3, \dots \}$. Assume further that  $p_{i} >0$ for $i \in \mathcal{I} \cup \{0,1\}$ and $n_{\kk}(1) = |\mathbf{n}_{\kk}| p_{1}+o\Bigpar{ \frac{|\mathbf{n}_{\kk}|}{\ln  |\mathbf{n}_{\kk}|}}$. Let  $b_{i}\in\bbNo$ for $i \in \mathcal{I}$ and let $T_{\kk} \in \mathbb{T}$ be such that 
\begin{align} \label{eqEj1new}
n_{T_{\kk}}(0) = \sum_{i \in \mathcal{I}} (i-1)b_{i}+1,\qquad
n_{T_{\kk}}(1) =   \frac{\ln |\mathbf{n}_{\kk}|}{-\ln p_{1}} +o(1),
\quad \quad   n_{T_{\kk}}(i) = b_{i}
\text{\quad for }i\in\cI,
\end{align}
and $n_{T_{\kk}}(i)=0$ for $i \not \in \mathcal{I}$. Note that this is possible for some
(but not all) sequences $\bn_\kk$.

 It follows from \eqref{eqEj1new} that 
$m_{\kappa} = |T_{\kk}|\ge n_{T_\kk}(1) \rightarrow \infty$, as $\kappa \rightarrow \infty$. 
To apply \refT{The01}\ref{ReGD10}, we only need
to verify that $|\bn_{\kk}| \pi_{\mathbf{p}(\bn_{\kk})}(T_{\kappa})
\rightarrow \lambda$ for some $\lambda \in (0,\infty)$.  

We obtain from \eqref{pip} and \eqref{eqEj1new}
\begin{align} \label{eqEj2new}
&\ln \left( |\bn_{\kk}| \pi_{\mathbf{p}(\mathbf{n}_{\kk})}(T_{\kk}) \right)  
= \ln |\mathbf{n}_{\kk}| + n_{T_{\kk}}(0) \ln p_{0}(\mathbf{n}_{\kappa}) + n_{T_{\kk}}(1) \ln p_{1}(\mathbf{n}_{\kappa}) + \sum_{i \in  \mathcal{I}}n_{T_{\kk}}(i) \ln p_{i}(\mathbf{n}_{\kappa})
\nonumber \\& \quad
= \ln |\mathbf{n}_{\kk}| 
+\Bigpar{\sum_{i \in \mathcal{I}} (i-1)b_{i}+1} \ln p_{0}(\mathbf{n}_{\kappa}) 
+n_{T_{\kk}}(1) \ln p_{1}(\mathbf{n}_{\kappa})
+ \sum_{i \in \mathcal{I}} b_{i} \ln p_{i}(\mathbf{n}_{\kappa}).
\end{align}
We have
\begin{align}\label{eqEj2.5new}
n_{T_{\kk}}(1) \ln p_{1}(\mathbf{n}_{\kappa}) & = -\frac{\ln |\mathbf{n}_{\kk}|}{\ln p_{1}}  \ln \left(\frac{|\mathbf{n}_{\kk}| p_{1}+o\Bigpar{ \frac{|\mathbf{n}_{\kk}|}{\ln  |\mathbf{n}_{\kk}|}}}{|\mathbf{n}_{\kk}|} \right) + o(1) \nonumber \\
& = -\frac{\ln |\mathbf{n}_{\kk}|}{\ln p_{1}} \ln \left(p_{1}+o\Bigpar{\frac{1}{\ln |\mathbf{n}_{\kk}|}} \right)  + o(1)
= -\ln |\mathbf{n}_{\kk}| +o(1)
\end{align}
and thus \eqref{eqEj2new} implies
\begin{align} \label{eqEj4new}
\ln \left( |\bn_{\kk}| \pi_{\mathbf{p}(\mathbf{n}_{\kk})}(T_{\kk}) \right) 
&= \ln |\mathbf{n}_{\kk}| 
+\Bigpar{\sum_{i \in \mathcal{I}} (i-1)b_{i}+1} \ln p_{0}
-\ln |\mathbf{n}_{\kk}|
+ \sum_{i \in \mathcal{I}} b_{i} \ln p_{i} +o(1)
\notag\\&
\to
\Bigpar{\sum_{i \in \mathcal{I}} (i-1)b_{i}+1}  \ln p_{0}
+ \sum_{i \in \mathcal{I}} b_{i} \ln p_{i}
.\end{align}
Hence, $|\bn_{\kk}| \pi_{\mathbf{p}(\bn_{\kk})}(T_{\kappa}) \rightarrow
\lambda$ with $\gl := p_{0}^{\sum_{i \in \mathcal{I}} (i-1)b_{i}+1} \prod_{i
  \in \mathcal{I}}p_{i}^{b_{i}}$, and thus \refT{The01}\ref{ReGD10} yields
$N_{T_\kk}(\cT_{\bn_\kk})\dto\Po(\gl)$.
\end{example}

\subsection{Stein's method with coupling}\label{SSstein2}

In this subsection we give an alternative proof of (most of) the results above,
using Stein's methods with couplings, see \cite{SJI} for a general description.
(The proof is not completely independent of the proof above, since we reuse
the calculation of the first two moments from \eqref{eq11G}.)
The main point of doing this is not to just give an alternative proof,
but to give the quantitative error bound \eqref{jm1}.

\begin{remark}
 In the results below, we approximate the distribution of
$N_T(\cTbn)$ with the Poisson distribution with the same mean $\E[N_T(\cTbn)]$,
instead of using its approximation $|\bn|\pi_{\mathbf{p}(\bn)}(T)$ as 
in \refTs{The01} and \ref{The1}. It follows from \eqref{jmb2} below,
see also \eqref{eq11G+},
that the difference usually is unimportant. 
In particular, 
using the notation below, 
if $\gD<\frac14\gl$ (which can be relaxed to $\gD=O(\gl)$), then
\eqref{jmb2} holds, and thus 
\eqref{jm1} holds also for  
$\dtv\bigpar{ N_T(\cTbn),\Po(|\bn|\pi_{\bp(\bn)}(T))}$.
See also \refR{Rbad?}.
\end{remark}

\begin{theorem}\label{Tstein}
Let $\bn$ be a degree statistic and $T\in\bbT$ a tree, 
and let $n:=|\bn|$ and $m:=|T|$.
There exists a universal constant $C$, not depending on $\bn$ or $T$, such 
that if\/ $\mathcal{T}_{\mathbf{n}}\simx {\rm Unif}(\mathbb{T}_{\mathbf{n}})$,
then, with $\gl:=\E N_T(\cTbn)$,
\begin{align}\label{jm1}
\dtv\bigpar{ N_T(\cTbn),\Po(\gl)}
\le C\gl\sum_{i\in\cD(T)}\frac{n_T(i)^2}{n(i)}
= 
C m^2\frac{\gl}{n}\sum_{i\in\cD(T)}\frac{p_i(\bn_T)^2}{p_i(\bn)}
.\end{align}
\end{theorem}

\begin{proof}
Let, 
recalling \eqref{pin},
\begin{align}\label{jm2}
\gD:=\gl\sum_{i\in\cD(T)}\frac{n_T(i)^2}{n(i)}
= 
 m^2\frac{\gl}{n}\sum_{i\in\cD(T)}\frac{p_i(\bn_T)^2}{p_i(\bn)},
\end{align}
so \eqref{jm1} can be written $\dtv\bigpar{ N_T(\cTbn),\Po(\gl)}\le C\gD$. Note that for any random variable $X$ with values in $\bbNo$, Markov's
inequality implies
$\dtv(X,0)=\P(X\ge1) \le \E X$.
Hence,
\begin{align}
  \label{jm3}
\dtv\bigpar{ N_T(\cTbn),\Po(\gl)}
\le
\dtv\bigpar{ N_T(\cTbn),0}
+
\dtv(0,\Po(\gl))
\le 2\gl.
\end{align}
Consequently, if $\gD\ge \frac{1}{4}\gl$, then \eqref{jm1} is trivial (for
any $C\ge8$). We may thus assume $\gD<\frac{1}{4}\gl$, which by \eqref{jm2}
means $\gl>0$ and
\begin{align}
  \label{jm4}
\sum_{i\in\cD(T)}\frac{n_T(i)^2}{n(i)}
<\tfrac14.
\end{align}
It follows that for every $i\ge0$, we have
\begin{align}
  \label{jm5}
n_T(i) \le n_T(i)^2 \le \tfrac14 n(i),
\end{align}
and by summing over $i\ge0$ we have also
\begin{align}
  \label{jm6}
m=|T| \le \tfrac14 n.
\end{align}

After these preliminaries, we come to the main part of the proof.
We assume that $\cTbn$ is constructed by \eqref{ctd} from a uniformly random
sequence $\bd=(d_1,\dots,d_n)$ in $\bbB_\bn$, see \refSS{TreesandDegrees},
and recall that $d_\ell$ is interpreted cyclically, with indices taken modulo $n$.
Define the random indicator variables, for $k=1,\dots,n$,
\begin{align} 
I_k=I_{k}(\mathbf{d},T) & 
\coloneqq \bigindic{(d_{k+i-1})_{i=1}^m=\bd_{T}}, \label{jm7}
\end{align} 
and
note that $I_k$ depends only on the degrees $d_i$ for
$i\in\cI_k:=\set{k,\dots,k+m-1}$. 
Then,  \eqref{b1} gives
\begin{align}\label{jm8}
  N_T(\cTbn)=W:=\sumkn I_k.
\end{align}
$\E I_k$ does not depend on $k$, by the rotational symmetry;
hence, for every $k\in[n]$,
\begin{align}\label{jm9}
\E I_k=\E[W]/n=\gl/n.
\end{align}

In order to use Stein's method with couplings, we also need to study the sum $W$
when we condition on a given $I_k$ to be 1; we do this using the following
construction.
(We assume that $n_T(i)\le n(i)$ for each $i$, as we may by \eqref{jm5};
note also that otherwise $\cTbn$ cannot contain any fringe tree equal to $T$, 
so $W=0$ and we cannot condition on $I_k=1$.) Fix $k\in[n]$ and 
take a random sequence $\bd\in\bbB_\bn$. Then
\begin{enumerate}[label=\upshape\arabic*.]
\item 
For each $i\ge0$, consider the $n(i)$
degrees $d_\ell$ in the sequence that equal $i$, and mark $n_T(i)$
of them, chosen uniformly at random. 
\item 
Remove the $m$ degrees $d_\ell$ for 
$\ell\in\cI_k$
(i.e., the degrees that determine $I_k$), 
and put them in a storage room.
\item \label{cop3}
Move the $m$ marked degrees, whether removed or not, to the now vacant 
positions in $\cI_k$, 
in the correct order such that they form a copy of $\bd_T$.
\item 
Return the remaining (unmarked) degrees from the storage, and put them in
the positions made vacant by \ref{cop3}, in uniformly random order.
\end{enumerate}
Denote the resulting degree sequence by $\bhdk=(\hd_i\xk)_1^n$.
(We again interpret the index modulo $n$.)  It is clear from the construction that $\bhdk\in\bbB_\bn$, and that the
distribution of $\bhdk$ equals the conditional distribution of $\bd$ given
$I_k=1$. (This means just that the degrees $d_\ell$ with $\ell\notin\cI_k$
are in uniformly random order.)
We define, for $\ell=1,\dots,n$,
\begin{align} \label{ari5}
J_{\ell k}
\coloneqq 
I_\ell(\bhdk,T)=
\bigindic{(\hdk_{\ell+i-1})_{i=1}^m=\bd_{T}},
\end{align} 
the indicator that $\bd_T$ appears at position $\ell$ in $\bhdk$.
(In particular, by definition, $J_{kk}=1$.)
Then, the distribution of the sequence $(J_{\ell k})_{\ell=1}^n$
equals the conditional distribution of 
$(I_\ell)_{\ell=1}^n$ given $I_k=1$.
This is the property required in \cite[(2.1.1)]{SJI}.
Hence it follows from \cite[(2.1.2), see also Theorem 2.A]{SJI} 
together with \eqref{jm9}
that
\begin{align}\label{ari8}
  \dtv\bigpar{W,\Po(\gl)}
\le \frac{1-e^{-\gl}}{\gl} \sumkn \frac{\gl}{n} 
\E\Bigabs{I_k +\sum_{\ell\neq k}(  I_\ell-J_{\ell k})}.
\end{align}
We ignore for simplicity the factor $1-e^{-\gl}$ (which is important only
for $\gl\to0$), and obtain from \eqref{ari8}
\begin{align}\label{ari9}
  \dtv\bigpar{W,\Po(\gl)}
\le \frac{1}{n} \sumkn \Bigpar{
\E I_k + \E\sum_{\ell\neq k}|  I_\ell-J_{\ell k}|}
=
\frac{\gl}{n} + \frac{1}{n} 
\E\sumkn \sum_{\ell\neq k}|I_\ell-J_{\ell k}|.
\end{align}
Let $\cA_k$ be the random set of indices $\ell\in[n]$ such that
the set of indices $\cI_\ell$ contains 
either the original position of some marked degree, or 
at least one of the indices in $\cI_k$.
If $\ell\notin\cA_k$, then the degrees 
$(\hd\xk_i,\,i\in\cI_\ell) = (d_i,\,i\in\cI_\ell)$, and thus
$J_{\ell k}=I_\ell$. Hence, recalling $J_{kk}=1$,
\begin{align}\label{ari10}
  \sumkn \sum_{\ell\neq k}|I_\ell-J_{\ell k}|
&= \sumkn\sum_{\ell\in\cA_k\setminus\set k}|I_\ell-J_{\ell k}|
\le  \sumkn\sum_{\ell\in\cA_k\setminus\set k}(I_\ell+J_{\ell k})
\notag\\&
\le  \sumkn\Bigpar{\sum_{\ell\in\cA_k}(I_\ell+J_{\ell k})-1}
\notag\\&
=
 \sumkn\sum_{\ell\in\cA_k}(J_{\ell k}-I_\ell)
+ 2  \sumkn\sum_{\ell\in\cA_k}I_\ell-n
\notag\\&
=
 \sumkn\sumln(J_{\ell k}-I_\ell)
-n
+ 2  \sumkn\sum_{\ell\in\cA_k}I_\ell.
\end{align}
We take the expectation, and note first that
(as in similar calculations in e.g.\ \cite[(2.1.4)]{SJI})
\begin{align}\label{tau1}
\frac{\gl}{n}  \sumkn\sumln\E J_{\ell k}&
=
\sumkn\P(I_k=1)\sumln \E [I_\ell\mid I_k=1]
=
\sumkn\sumln \E [I_\ell I_k]
=\E [W^2],
\\
\frac{\gl}{n}  \sumkn\sumln \E I_{\ell}&
=
n^2\Bigpar{\frac{\gl}{n}}^2
=(\E W)^2,
\end{align}
and thus
\begin{align}\label{tau2}
\frac{\gl}{n}\E \sumkn\sumln(J_{\ell k}-I_\ell)
=\E[W^2]-(\E W)^2
=\Var W.
\end{align}
For the final sum in \eqref{ari10}, we note that for each $k$,
we have $2m-1$ indices $\ell$ for which $\cI_\ell\cap\cI_k\neq\emptyset$;
these give the contribution
\begin{align}\label{tau3}
 \E \sumkn\sum_{|\ell-k|<m} I_{\ell}
= n(2m-1)\frac{\gl}{n} =(2m-1)\gl.
\end{align}
For any other $\ell$, we have $\ell\in \cA_k$ if and only if one of the
degrees in $\cI_\ell$ is marked.
For each $i$, the degree sequence $\bd$ contains $n(i)$ degrees that are
equal to $i$, 
and $n_T(i)$ of these will be marked.
Conditioned on $I_\ell=1$, we have $n_T(i)$ degrees $i$ in $\cI_\ell$, 
and thus the probability that at least one of them is marked is at most
\begin{align}\label{tau4}
  \sum_{j=1}^{n_T(i)}\frac{n_T(i)}{n(i)-j+1} \le \frac{n_T(i)^2}{n(i)-n_T(i)+1}.
\end{align}
(An exact expression is easily given, using a hypergeometric distribution,
but we have no need for it.)
Using our assumption \eqref{jm5}
we thus have,
for $k$ and $\ell$ such that $\cI_k\cap\cI_\ell=\emptyset$,
\begin{align}\label{tau5}
  \E [\indic{\ell\in \cA_k}I_\ell]
= \frac{\gl}{n} \E [\indic{\ell\in \cA_k}\mid I_\ell=1]
\le \frac{\gl}{n} \sum_{i\in \cD(T)} 2\frac{n_T(i)^2}{n(i)}
=\frac{2\gD}{n}.
\end{align}
There are $n(n-2m+1)\le n^2$ such terms, and combining \eqref{ari10},
\eqref{tau2}, \eqref{tau3}, and \eqref{tau5}, we obtain
\begin{align}
  \label{tau6}
\frac{1}{n} \E \sumkn \sum_{\ell\neq k}|I_\ell-J_{\ell k}|
\le \frac{\Var W}{\gl} - 1 + \frac{2}{n}\Bigpar{(2m-1)\gl+ 2n\gD}.
\end{align}
Consequently, \eqref{ari9} yields, recalling $\E W = \gl$, 
\begin{align}\label{tau7}
  \dtv\bigpar{W,\Po(\gl)}&
\le \frac{\gl}{n} + 
\frac{\Var W}{\gl} - 1 + 2(2m-1)\frac{\gl}{n}+4\gD
\notag\\&
\le 4m\frac{\gl}{n} + \frac{\Var[W]-\E[W]}{\E[W]}
+4\gD.
\end{align}
By \eqref{rny0} and \eqref{jm2}, we have
\begin{align}
  \label{tau8}
m \frac{\gl}{n} \le m^2 \frac{\gl}{n} 
\le 8 m^2 \frac{\gl}{n} \sum_{i\in\cD(T)}\frac{p_i(\bn_T)^2}{p_i(\bn)}
=8\gD.
\end{align}
Furthermore, 
our assumptions \eqref{jm5}--\eqref{jm6} show that \eqref{eq4G} applies
and that for $q=1,2$, as in \eqref{eq11G}, 
\begin{align}\label{tau9}
\E (N_{T}(\mathcal{T}_{\mathbf{n}}))_{q}&   
= n^{-q m+q}\prod_{i\in\cD(T)} n(i)^{q n_{T}(i)}
\cdot \exp \Bigpar{
O\Bigpar{\frac{q^{2}m^{2}}{n } + 
\sum_{i\in \cD(T)} \frac{q^{2}n_{T}(i)^{2}}{n(i)}} }.
\end{align}
We have, by \eqref{tau8}, $m^2/n \le 8\gD/\gl$,
which together with \eqref{jm2} shows that the exponent in \eqref{tau9} is
$O(q^2\gD/\gl)=O(\gD/\gl)$.
Hence, recalling the notations \eqref{pin} and \eqref{pip},
\eqref{tau9} can be written as
\begin{align}\label{tau10}
\E (N_{T}(\cTbn))_{q}&   
=(n\pi_{\bp(\bn)}(T))^q e^{O(\gD/\gl)},
\qquad q=1,2.
\end{align}
For $q=1$, this is
\begin{align}\label{tau11}
\gl 
= \E N_{T}(\cTbn)   
=n\pi_{\bp(\bn)}(T)  e^{O(\gD/\gl)}.  
\end{align}
Since we have assumed $\gD<\gl/4$, this implies
\begin{align}\label{jmb1}
  n\pi_{\bp(\bn)}(T) = \Theta(\gl).
\end{align}
Consequently, \eqref{tau10} implies
\begin{align}\label{jmb2}
\gl=\E N_{T}(\cTbn)&   
=n\pi_{\bp(\bn)}(T)+O(\gD),
\\\label{jmb3}
\E (N_{T}(\cTbn))_{2}&   
=(n\pi_{\bp(\bn)}(T))^2 + O(\gl\gD),
\end{align}
and thus, using again the notation $W=N_{T}(\cTbn)$,
\begin{align}\label{jmb4}
  \Var W = \E[(W)_2] + \E W - (\E W)^2
= \E W + O(\gl\gD + \gD^2)
= \E W + O(\gl\gD).
\end{align}
Consequently,
\begin{align}\label{jmb5}
  \frac{\Var W - \E W}{\E W} 
= O(\gD).
\end{align}
Finally, \eqref{tau7} together with \eqref{tau8} and \eqref{jmb5}
implies
\begin{align}\label{jmb6}
  \dtv\bigpar{N_{T}(\cTbn),\Po(\gl)}
= O(\gD),
\end{align}
which is the result \eqref{jm1}.
\end{proof}

\begin{remark} \label{RemarkSteinConst}
With more care in the estimates in the proof of Theorem \ref{Tstein}, it can
be shown that the constant $C$ can be chosen as 
$60 e^{\xfrac{29}{4}} <85\,000$. 
( We have not tried to optimize the calculations, and 
this is not the best possible bound.) The calculations are given in Appendix \ref{ConstStein}.
\end{remark}

\begin{corollary} \label{Cstein}
  Assume \refCond{Condition1}. 
For $\kk \geq 1$, let $T_{\kappa} \in \mathbb{T}$ be such that
as \kktoo,
with $m_{\kappa} \coloneqq |T_{\kappa}|$,
\begin{align}
\label{rny2s}
\frac{m_\kk^2\ENTTk}{|\bn_\kk|}\cdot\!\!\!
\sum_{i\in\cD(T_\kk)}\frac{p_i(\bn_{T_\kk})^2}{p_i(\bn_\kk)}
=o(1).
\end{align}
Then:
\begin{romenumerate}[-20pt] 
\item \label{ReGD1S} 
If\/ $\ENTTk \rightarrow \lambda$, 
for some $\lambda \in [0,\infty)$, then 
\begin{align} \label{regd1s}
N_{T_\kk}(\cTbnk) \dto\Po\bigpar{\gl}.
\end{align} 
\item \label{ReGD2S} 
If\/ $\ENTTk \rightarrow \infty$, then 
\begin{align}\label{regd2s}
\frac{N_{T_\kk}(\cTbnk)- \ENTTk}{\sqrt{\ENTTk}} \dto\N\bigpar{0,1}.
\end{align}
\end{romenumerate}  
\end{corollary}
\begin{proof}
  Let $\gl_\kk:=\ENTTk$. Then \eqref{jm1} shows, using the assumption
  \eqref{rny2s},  that
  \begin{align}\label{jmb7}
\dtv\bigpar{N_{T_\kk}(\cTbnk),\Po(\gl_\kk)}\to0.    
  \end{align}

In \ref{ReGD1S}, we have $\Po(\gl_\kk)\dto\Po(\gl)$, and \eqref{regd1s}
follows from \eqref{jmb7}.

Similarly,
in \ref{ReGD2S}, \eqref{regd2s} follows from \eqref{jmb7} and the central
limit theorem for the Poisson distributions $\Po(\gl_\kk)$.
\end{proof}

\subsection{The method of Cai and Devroye}\label{SSCai}

Cai and Devroye \cite{Cai2017} study the related problem of fringe trees in
a conditioned Galton--Watson tree.
They note that for any fixed tree $T$, conditioned on the number 
of the fringe trees in the \cGWt\ that
have the same size as $T$, i.e.\ belong to $\bbT_{|T|}$,
these fringe trees are independent and uniformly distributed elements of
$\bbT_{\bn_T}$. 
Hence, the number of these fringe trees that are equal to $T$ has 
conditionally a binomial distribution.

The same also holds for our random trees $\cTbn$ with a given degree statistic
$\bn$ 
if we condition on 
the number of the fringe trees that 
have the same degree statistic as $T$, i.e.\ belong to $\bbT_{\bn_T}$;
then
these fringe trees are independent and uniformly distributed elements of
$\bbT_{\bn_T}$, and consequently,
the number of these fringe trees that are equal to $T$ has 
conditionally a binomial distribution too, and thus, conditioned on $N_{\bn_T}(\cTbn)$ (defined in \eqref{eq2s}), 
we have
\begin{align}\label{luc1}
N_T(\cTbn) \simx \Bi\bigpar{N_{\bn_T}(\cTbn),1/|\bbT_{\bn_T}|}.  
\end{align}

We can thus use the method by Cai and Devroye \cite{Cai2017},
who showed the following general result.
\begin{lemma}[{\cite[Lemma 2.10]{Cai2017}}]
Let $X$ and $M$ be non-negative integer-valued random variables.  
If conditioned on the event $M=m$, $X$ is binomial $(m,p)$, then
\begin{align}\label{luc2}
  \dtv\bigpar{X,\Po(\E X)}
\le p + \sqrt{p\frac{\Var M}{\E M}}.
\end{align}
\end{lemma}

Combining \eqref{luc1} and \eqref{luc2}, we find, letting again
$\gl := \ENTT$,
\begin{align}\label{luc3}
\dtv\bigpar{N_T(\cTbn),\Po(\gl)}
\le \frac{1}{|\bbT_{\bn_T}|}
+ \sqrt{\frac{\Var [N_{\bn_T}(\cTbn)]}{|\bbT_{\bn_T}|\cdot\E[N_{\bn_T}(\cTbn)]}}.
\end{align}
Here the mean and variance of $N_{\bn_T}(\cTbn)$ enter. These can easily be
expressed using the mean and variance of $N_T(\cTbn)$ as follows.
\begin{lemma} \label{LCorollary1}
Let $\mathbf{n}$ be a degree statistic and let $\mathcal{T}_{\mathbf{n}} \simx {\rm Unif}(\mathbb{T}_{\mathbf{n}})$. For $T \in \mathbb{T}$,
\begin{align}
\mathbb{E} N_{\mathbf{n}_{T}}(\mathcal{T}_{\mathbf{n}}) & = |\mathbb{T}_{\mathbf{n}_{T}}|  \cdot  \mathbb{E} N_{T}(\mathcal{T}_{\mathbf{n}}),  \label{eq5g} \\
{\rm Var}(N_{\mathbf{n}_{T}}(\mathcal{T}_{\mathbf{n}})) & 
=  |\mathbb{T}_{\mathbf{n}_{T}}|^{2}  \cdot  {\rm Var}(N_{T}(\mathcal{T}_{\mathbf{n}})) - |\mathbb{T}_{\mathbf{n}_{T}}| \cdot (|\mathbb{T}_{\mathbf{n}_{T}}|-1) \cdot  \mathbb{E} N_{T}(\mathcal{T}_{\mathbf{n}})
\notag\\&
=|\bbT_{\bn_T}|^2\cdot\bigpar{\Var N_T(\cTbn)-\E N_T(\cTbn)}
+ |\bbT_{\bn_T}|\cdot\E N_T(\cTbn)
.\label{eq6g}
\end{align} 
\end{lemma}

\begin{proof}
\cite[Lemma 3.1]{SJ378} shows
that $\E N_{T}(\mathcal{T}_{\mathbf{n}})$ depends only on $\bn$ and
the degree statistic of $T$, and thus,
for any $T^{\prime} \in \mathbb{T}_{\mathbf{n}_{T}}$,
$\E N_{T^{\prime}}(\mathcal{T}_{\mathbf{n}}) = \E N_{T}(\mathcal{T}_{\mathbf{n}})$.
Hence, \eqref{eq5g} follows from \eqref{eq2s}.
Furthermore,  \eqref{eq2s} yields also
\begin{align} \label{eq8g}
(N_{\mathbf{n}_{T}}(\mathcal{T}_{\mathbf{n}}))^{2} = \sum_{T^{\prime} \in \mathbb{T}_{\mathbf{n}_{T}}} (N_{T^{\prime}}(\mathcal{T}_{\mathbf{n}}))^{2} + \sum_{T^{\prime} \in \mathbb{T}_{\mathbf{n}_{T}}} \sum_{T^{\prime \prime} \in \mathbb{T}_{\mathbf{n}_{T}}} N_{T^{\prime}}(\mathcal{T}_{\mathbf{n}}) N_{T^{\prime \prime}}(\mathcal{T}_{\mathbf{n}}) \mathbf{1}_{\{ T^{\prime}\neq T^{\prime \prime}\}}.
\end{align}
If
$T^{\prime},T^{\prime \prime} \in \mathbb{T}_{\mathbf{n}_{T}}$
such that $T^{\prime}\neq T^{\prime \prime}$, 
then two different fringe trees in $\cT_n$ that are isomorphic to $T'$
have to be disjoint, and the same holds for 
two fringe trees isomorphic to $T'$ and $T''$, respectively.
It follows easily by the arguments in the proof of \cite[Lemma 3.1]{SJ378},
see also \cite[Lemmas 3.2 and~3.3]{SJ378} for more general and more detailed
results, 
that
$\E [(N_{T^{\prime}}(\mathcal{T}_{\mathbf{n}}))^{2}]
= \E[(N_{T}(\mathcal{T}_{\mathbf{n}}))^{2}]$ and 
$\E[N_{T^{\prime}}(\mathcal{T}_{\mathbf{n}})
N_{T^{\prime\prime}}(\mathcal{T}_{\mathbf{n}})] = \E [(N_{T}(\mathcal{T}_{\mathbf{n}}))_{2}]$. 
Hence, it follows
from \eqref{eq8g} that
\begin{align} \label{eq9g}
\E\bigsqpar{ N_{\mathbf{n}_{T}}(\mathcal{T}_{\mathbf{n}})^{2} }
= |\mathbb{T}_{\mathbf{n}_{T}}|
  \mathbb{E}\bigsqpar{N_{T}(\mathcal{T}_{\mathbf{n}})^{2}} +
  |\mathbb{T}_{\mathbf{n}_{T}}| \cdot (|\mathbb{T}_{\mathbf{n}_{T}}|-1)
  \cdot
  \bigpar{\mathbb{E}\bigsqpar{N_{T}(\mathcal{T}_{\mathbf{n}})^{2}}
-\E \bigsqpar{N_T(\cT_\bn)}} ,
\end{align}
and \eqref{eq6g} follows. 
\end{proof}

Using \eqref{luc3} and \refL{LCorollary1}, we obtain
\begin{align}\label{luc4}
\dtv\bigpar{N_T(\cTbn),\Po(\gl)}&
\le \frac{1}{|\bbT_{\bn_T}|}
+ \sqrt{\frac{\Var [N_{T}(\cTbn)]-\E N_T(\cTbn)}{\ENTT}+\frac{1}{|\bbT_{\bn_T}|}}
\notag\\&
\le
\sqrt{\frac{\Var [N_{T}(\cTbn)]-\E N_T(\cTbn)}{\ENTT}}
+\frac{2}{\sqrt{|\bbT_{\bn_T}|}}
.\end{align}
Using \eqref{jmb5} from the proof above, we obtain from \eqref{luc4}
\begin{align}\label{luc5}
\dtv\bigpar{N_T(\cTbn),\Po(\gl)}&
=O\bigpar{\gD\qq+|\bbT_{\bn_T}|\qqw}
.\end{align}
(The proof of \eqref{jmb5} assumes $\gD<\frac{1}{4}\gl$, but \eqref{luc5}
holds also for $\gD\ge\frac14\gl$ by \eqref{jm3}.)
This is weaker than \eqref{jm1}, but enough to prove \refC{Cstein} provided
$|\bbT_{\bn_{T_\kk}}|\to\infty$, which can be shown to hold provided
$|T_\kk|\to\infty$ except for the two cases when $T_\kk$ is a path or a star.
(In these exceptional cases,  $|\bbT_{\bn_{T_\kk}}|=1$, and the method is
useless.)

For our purposes, the method by Cai and Devroye \cite{Cai2017} thus yields 
a simpler proof of part of our main results, but there are exceptional cases
(including the case when $T_\kk=T$ is fixed)
where it does not yield convergence in distribution.
Moreover, the bound \eqref{luc5} for the total variation distance is
inferior to the bound \eqref{jm1} obtained by Stein's method.
Nevertheless, we include the method here since it is simple, and it might be
useful in other situations too, and we find it interesting to compare the
result of it with the results from other methods.

\subsection{Stein's method with exchangeable pairs}\label{SSstein1}

There are many versions of Stein's method, both for Poisson approximation
and normal approximation, see e.g.\ \cite{Chen2011} and \cite{Ross2011}.
We have in \refSS{SSstein2} used Stein's method with couplings.
Other popular versions use exchangeable pairs.
Cai and Devroye used, in a preliminary version \cite{Cai2017arxiv}
of the paper \cite{Cai2017} used above,
a version of
Stein's method with exchangeable pairs for Poisson approximation,
which was
developed by Chatterjee, Diaconis, and Meckes \cite{Diaconis2005} and
further discussed in Ross \cite[Section 4.4]{Ross2011}. 
Recall that a pair $(X,X^{\prime})$
of random variables is called an exchangeable pair if  
$(X,X^{\prime}) \eqd (X^{\prime},X)$. 
We sketch one way to construct an exchangeable pair in our setting.

Let $X:=N_T(\cTbn)$, where $T$ is a given tree and $\bn$ is a given degree
sequence.
Let $V$ be a uniformly random vertex of $\cTbn$,
and let $\cT'$ be a uniformly random tree in $\bbT_{\bn_T}$, independent
of $\cTbn$ and $V$.
If the fringe tree $(\cTbn)_V$ has the same degree statistic $\bn_T$ as $T$,
then define $\cTbn'$ as the tree $\cTbn$ with this fringe tree
$(\cTbn)_V$ replaced by $\cT'$; otherwise, let $\cTbn'=\cTbn$.
Note that $\cTbn'$ has the same degree statistic $\bn$ as $\cTbn$.
Moreover, it is easy to see that $(\cTbn,\cTbn')$ is an exchangeable pair of
random trees. 
Hence,
$(N_T(\cTbn),N_T(\cTbn'))$ is an exchangeable pair of random variables.

Using this exchangeable pair and
\cite[Theorem 4.5 and equation (4.14)]{Ross2011}, 
see also \cite[Lemma~2 and Proposition~3]{Diaconis2005},
it can be shown after lengthy calculations that
\begin{align}\label{theos}
d_{\rm TV}\bigpar{N_{T}(\mathcal{T}_{\mathbf{n}}), {\rm Po}(\mathbb{E}[N_{T}(\mathcal{T}_{\mathbf{n}})])} 
& \leq \frac{1}{|\mathbb{T}_{\mathbf{n}_{T}}|} 
+\left(\frac{{\rm Var}(N_{T}(\mathcal{T}_{\mathbf{n}})) 
 -\mathbb{E} N_{T}(\mathcal{T}_{\mathbf{n}})}
 {\mathbb{E}N_{T}(\mathcal{T}_{\mathbf{n}})} 
 + \frac{1}{|\mathbb{T}_{\mathbf{n}_{T}}|} \right)^{1/2}
\notag\\&
\le  \frac{2}{|\mathbb{T}_{\mathbf{n}_{T}}|\qq} 
+\left(\frac{{\rm Var}(N_{T}(\mathcal{T}_{\mathbf{n}})) 
 -\mathbb{E} N_{T}(\mathcal{T}_{\mathbf{n}})}
 {\mathbb{E}N_{T}(\mathcal{T}_{\mathbf{n}})}  \right)^{1/2}
.\end{align}
Note that this is exactly the same estimate as \eqref{luc4}, obtained by a
much simpler method. We therefore find (as apparently the authors of
\cite{Cai2017} did for their problem) that in our case, this version of
Stein's method works, but it is not really useful since the simpler method
in \refSS{SSCai} yields the same result in a simpler way.
We therefore omit the calculations leading to \eqref{theos}.
(This does not exclude the possibility that better results might be obtained
in the future
by this method using some other exchangeable pair, or the same pair and
different estimates in the calculations.)

\section{Fringe trees with a given degree statistic 
(depending on  $\kk$)}\label{Sstatistic}

For $T\in \mathbb{T}$ and a degree statistic $\bn$, 
recall that we have defined $N_{\bn}(T)$ in \eqref{eq2s} 
to be the number of fringe trees of $T$ with degree statistic $\bn$. 

\begin{lemma} \label{lemmaMoment}
Let $\mathbf{n}$ and $\bm$ be two degree statistics and let
$\mathcal{T}_{\mathbf{n}} \simx {\rm Unif}(\mathbb{T}_{\mathbf{n}})$. 
Then, for every  $q \in \mathbb{N}$, 
we have
\begin{align} \label{sw2}
\E[(N_{\bm}(\mathcal{T}_{\mathbf{n}}))_{q}] = |\mathbb{T}_{\bm}|^{q} \frac{|\mathbf{n}|}{(|\mathbf{n}|)_{q|\bm|-q+1}} \prod_{i\geq 0} (n(i))_{qm(i)},
\end{align}
where we
(in the case $|\mathbf{n}|< q |\bm| -q +1$)
interpret $\frac00=0$.
Equivalently, for any tree $T$ with degree statistic $\bm$, we have
\begin{align} \label{sw2t}
\E[(N_{\bm}(\mathcal{T}_{\mathbf{n}}))_{q}] = |\mathbb{T}_{\bm}|^{q} 
\E[(N_{T}(\mathcal{T}_{\mathbf{n}}))_{q}].
\end{align}

\end{lemma}

\begin{proof}
For $q \in \mathbb{N}$, $(N_{\bm}(\mathcal{T}_{\mathbf{n}}))_{q}$ is the
number of sequences of $q$ distinct fringe subtrees of
$\mathcal{T}_{\mathbf{n}}$ that belong to
$\mathbb{T}_{\bm}$. 
Note that these fringe trees necessarily are disjoint (since they have the
same size). In particular, there are no such sequences 
if  $q m(i)> n(i)$ for some $i \geq 0$.
We may thus assume
$n(i)\ge qm(i)$ for every $i\ge0$, since otherwise \eqref{sw2} holds
trivially as $0=0$.
As a consequence, $|\bn|\geq|\bm|$.

Given such a sequence of $q$ fringe subtrees,
we say that these fringe subtrees are \emph{marked}. 
For a sequence $T_{1}, \dots, T_{q} \in
\mathbb{T}_{\bm}$, let $S(T_{1}, \dots, T_{q})$ be the number
of sequences of $q$ disjoint fringe trees in $\cTbn$ 
that are  copies of $T_{1}, \dots, T_{q}$, respectively. 
We thus have 
\begin{align}\label{sw3}
(N_{\bm}(\mathcal{T}_{\mathbf{n}}))_{q} = \sum_{T_{1}, \dots, T_{q} \in \mathbb{T}_{\bm}} S(T_{1}, \dots, T_{q}).
\end{align}
It is shown in \cite[equation (3.13)]{SJ378}
(in greater generality and with slightly different notation, see
(3.11)--(3.12) there) 
that for any sequence $T_1,\dots,T_q\in\bbT_\bn$, 
the number of trees in $\mathbb{T}_{\mathbf{n}}$ with marked sequences of
disjoint fringe subtrees that are equal (i.e.\ isomorphic) to $T_1,\dots,T_q$ is
given by
\begin{align} \label{sw4}
 \frac{(|\mathbf{n}| - q|\bm|+q-1)!}{\prod_{i \geq 0}(n(i) - q m(i))!}.
\end{align}
\noindent By dividing with $|\mathbb{T}_{\mathbf{n}}|$, which is given by
\eqref{sw1}, 
we find
(cf.\ the special case in \eqref{eq4G} where $T_1=\dots=T_q$)
\begin{align} \label{sw5}
\E[S(T_{1}, \dots, T_{q})] 
= \frac{|\bn|}{(|\mathbf{n}|)_{q|\bm|-q+1}} \prod_{i \geq 0}(n(i))_{q m(i)},
\end{align}
\noindent and our claim \eqref{sw2} follows from \eqref{sw3}, 
since there are $|\bbT_\bm|$ choices for each $T_i$.

We obtain \eqref{sw2t} from \eqref{sw2} and \eqref{eq4G}, 
or by \eqref{sw3} and noting that
we have shown that $\E S(T_1,\dots,T_q)$ does not depend on the choice of
$T_1,\dots,T_q$, and thus 
\begin{align}\label{sw6}
\E S(T_1,\dots,T_q)=\E S(T,\dots,T)
=\E[(N_{T}(\mathcal{T}_{\mathbf{n}}))_{q}].
\end{align}
\end{proof}

For a probability distribution $\mathbf{p} = (p_{i})_{i \geq  0}$ on $\mathbb{N}_{0}$, and a given set $\sT \subset \mathbb{T}$ of trees, we let
\begin{align}\label{pip2}
\pi_{\mathbf{p}}(\sT) \coloneqq  
\P(\cT_\bp\in\sT)
=\sum_{T \in \sT} \pi_{\mathbf{p}}(T)
.\end{align}
Note that for a given degree statistic $\bm$, the probability $\pi_\bp(T)$
defined in \eqref{pip} is the same for all $T\in\bbT_\bm$.
We denote it by
\begin{align}\label{pip3}
  \pi_\bp(\bm):=\prod_{i\ge0}p_i^{m(i)}.
\end{align}
Hence, 
\begin{align} \label{pip4}
\pi_\bp(\bbT_\bm)=|\bbT_\bm|\pi_\bp(\bm).
\end{align}

We have the following analogue of \refT{The1}, where we also let
$\cD(\bm):=\set{i:m(i)>0}$. 
\begin{theorem} \label{ThPo2} 
Assume \refCond{Condition1}. Let $\bn_\kk$ and $\bm_{\kappa}$, $\kk\ge1$, be
degree statistics such that 
\begin{align}
\label{rny2m}
&|\bm_\kk|^2\pi_{\bp(\bn_{\kk})}(\bbT_\bmk)\sum_{i\in\cD(\bm_\kk)}
\frac{p_i(\bm_{\kk})^2}{p_i(\bn_\kk)}
=o(1).
\end{align}
Then, as $\kk \rightarrow \infty$:
\begin{romenumerate}[-20pt]  
\item \label{ReGD1m} 
If\/ $|\bn_{\kk}|\pi_{\mathbf{p}(\bn_{\kk})}(\bbT_\bmk)\to\lambda$
for some $\lambda \in (0,\infty)$, then 
\begin{align} \label{regd1m}
N_{\bm_\kk}(\cT_\bnk) \dto\Po\bigpar{\gl}
\end{align} 
with convergence of all moments.

\item \label{ReGD2m} 
If\/ $|\bn_{\kk}| \pi_{\mathbf{p}(\bn_{\kk})}(\bbT_\bmk) \rightarrow \infty$,
then 
\begin{align}\label{regd2m}
\frac{N_{\bm_\kk}(\cTbnk)- |\bn_{\kk}| \pi_{\mathbf{p}(\bn_{\kk})}(\bbT_\bmk)}{\sqrt{|\bn_{\kk}| \pi_{\mathbf{p}(\bn_{\kk})}(\bbT_\bmk)}} \dto\N\bigpar{0,1},
\end{align}
\end{romenumerate}  
with convergence of mean and variance.

Thus, in both cases,
\begin{align}\label{regd3m}
\E\bigsqpar{N_{\bm_\kk}(\cTbnk)}
\sim
\Var\bigsqpar{N_{\bm_\kk}(\cTbnk)}
\sim
|\bn_{\kk}| \pi_{\mathbf{p}(\bn_{\kk})}(\bbT_\bmk).   
\end{align}
\end{theorem}

\begin{proof}
We argue as in the proof of \refT{The1}, with minor differences.
Note that \eqref{rny2m} and \eqref{ari3} imply, similarly to
\eqref{rny0}--\eqref{rny1},
\begin{align}\label{rny1m}
 &|\bm_{\kk}|^{2}  \pi_{\mathbf{p}(\bn_{\kk})}(\bbT_\bmk)=o(1).
\end{align}
Choose $T_\kk\in\bbT_{\bm_\kk}$, and let as in \eqref{glkk}
$\gl_\kk:=|\bn_{\kk}|\pi_{\mathbf{p}(\bn_{\kk})}(T_{\kappa})
=|\bn_{\kk}|\pi_{\mathbf{p}(\bn_{\kk})}(\bm_{\kappa})$.  
Define further 
\begin{align}\label{glkkm}
  \hgl_\kk:=|\bn_{\kk}|\pi_{\mathbf{p}(\bn_{\kk})}(\bbT_\bmk)
=|\bbT_{\bm_\kk}|\gl_\kk.
\end{align}
Hence we assume 
$\hgl_\kk\to\gl\le\infty$,
where we define $\gl:=\infty$ in Case \ref{ReGD2m}.
Let
$q_{\kk} \in \mathbb{N}$ be such that 
\begin{align}  \label{ny3m}
q_{\kk} = O\bigpar{\hgl_\kk\qq}.
\end{align}
Then \eqref{ny3m}, \eqref{glkkm}, \eqref{rny1m}, and \eqref{rny2m} imply that
\begin{align} \label{N1em}
\frac{q^{2}_{\kk}|\bm_{\kk}|^{2}}{|\mathbf{n}_{\kk}|} &
\le C
|\bm_{\kk}|^{2}  \pi_{\mathbf{p}(\bn_{\kk})}(\bbT_\bmk)
= o(1),
\\\label{N1fm}
\sum_{i\in \cD(\bm_{\kappa})} \frac{q_{\kk}^{2}m_{\kappa}(i)^{2}}{n_{\kappa}(i)}
&\le C |\bm_\kk|^2\pi_{\mathbf{p}(\bn_{\kk})}(\bbT_\bmk)
\sum_{i\in \cD(\bm_{\kappa})} \frac{p_i(\bm_{{\kappa}})^{2}}{p_i(\bn_{\kk})}
=o(1)
.\end{align}
It follows again that, for $\kk$ large enough, we have
$q_\kk |\bm_\kk|\le |\bn_\kk|/2$
and
$q_\kk m_{\kk}(i) \le n_\kk(i)/2$ for all $i\in\cD(\bm_\kk)$,
and then \eqref{eq4G} and \eqref{eq11G} hold (with $m_\kk:=|\bm_\kk|$).
Consequently, by \eqref{sw2t}, \eqref{eq11G}, and \eqref{glkkm},
\begin{align}
\E (N_{\bm_{\kappa}}(\mathcal{T}_{\mathbf{n}_{\kappa}}))_{q_{\kk}}&   
=
|\bbT_{\bm_{\kappa}}|^{q_\kk}\E (N_{T_{\kappa}}(\mathcal{T}_{\mathbf{n}_{\kappa}}))_{q_{\kk}}
\notag\\&
=|\bbT_{\bm_{\kappa}}|^{q_\kk} \gl_\kk^{q_{\kk}}  \cdot  \exp \bigpar{o(1)}
= \hgl_\kk^{q_{\kk}}  \cdot  \exp \bigpar{o(1)}
\label{eq11Gm}
.\end{align}
The conclusions follow as in the proof of \refT{The1} with only notational
changes (e.g.\ replacing $\gl_\kk$ by $\hgl_\kk$).
\end{proof}

\begin{remark}\label{Rbad2}
We have excluded the case $\gl=0$ from \refT{ThPo2}
for the same reason as for \refT{The1}, discussed in \refR{Rbad?}:
we have not been able to show 
(without additional conditions)
that
$|\bn_{\kk}|\pi_{\mathbf{p}(\bn_{\kk})}(\bbT_\bmk)\to0$
and \eqref{rny2m} imply $\E N_{\bmk}(\cTbnk)\to0$.
\end{remark}

We can also use Stein's method as in \refSS{SSstein2} and obtain the
following analogue of \refT{Tstein}. (The only difference is that 
$T$ and $n_T$ are replaced by $\bm$.)

\begin{theorem}\label{Tsteinm}
Let $\bn$ and $\bm$ be two degree statistics.
There exists a universal constant $C$, 
not depending on $\bn$ or $\bm$, such 
that if\/ $\mathcal{T}_{\mathbf{n}}\simx {\rm Unif}(\mathbb{T}_{\mathbf{n}})$,
then, with $\gl:=\E N_\bm(\cTbn)$,
\begin{align}\label{jm1m}
\dtv\bigpar{ N_\bm(\cTbn),\Po(\gl)}
\le C\gl\sum_{i\in\cD(\bm)}\frac{m(i)^2}{n(i)}
= 
C\gl \frac{|\bm|^2}{|\bn|}\sum_{i\in\cD(\bm)}\frac{p_i(\bm)^2}{p_i(\bn)}
.\end{align}
\end{theorem}

\begin{proof}
Let $t:=|\bbT_{\bm}|$, and
note that the new $\gl$ here is $t$ times the old $\gl$ in \refT{Tstein}, for
any $T \in \bbT_{\bm}$, 
by \eqref{eq5g}.
We argue almost exactly
as in the proof of \refT{Tstein}, with only minor differences.
Instead of \eqref{jm7}, we now let $I_k$ be the indicator that 
$(d_{k+i-1})_{i=1}^m\in\bbB_\bm$. 
(And similarly for $J_{\ell k}$ in \eqref{ari5}.)
We define the coupling as before, except
  that in Step 3, we put the marked degrees in uniformly 
random order in $\cI_k$.
Then the calculations up to \eqref{tau7} and \eqref{tau8} are the same as
before. To estimate the crucial term $(\Var W-\E W)/\E W$ in \eqref{tau7},
we note that \refL{LCorollary1} implies that, for any $T \in \bbT_{\bm}$,
\begin{align}\label{gem1}
  \frac{\Var N_{\bm}(\cTbn) - \E N_{\bm}(\cTbn)}{\E N_{\bm}(\cTbn)}
= |\bbT_{\bm}|  
\frac{\Var N_{T}(\cTbn) - \E N_{T}(\cTbn)}{\E N_{T}(\cTbn)},
\end{align}
so this term is now $|\bbT_{\bm}|=t $ times larger than in the proof of
\refT{Tstein}. On the other hand, as remarked above, 
$\gl$ is also $t$ times larger, and thus
$\gD$ is $t$ times larger now.
Hence, the estimate \eqref{jmb5} that was proved in the proof of \refT{Tstein}
holds in the present setting too, and the theorem follows.
\end{proof}

\begin{remark}
  We can similarly treat the number of fringe trees that belong to some
  other given subset of $\bbT_{\bnk}$, for example the number of fringe
  trees that are isomorphic to $T_\kk$ as unordered rooted trees.
(Note that a version of \refL{LCorollary1} holds for any subset of
$\bbT_{\bn_T}$, with the same proof.)
This gives straightforward generalizations of \refTs{ThPo2} and \ref{Tsteinm}.
We omit the details.
\end{remark}

\section{Fringe trees of a given size (depending on $\kk$)} 
\label{Ssize}

In this section, we will need \refCond{Cond2}.
Recall the notation $N_m(T)$ defined in \eqref{eq2m}.
Summing \eqref{b1} over all trees $T$ with $|T|=m$, we obtain 
(for $|\bn|\ge m$)
\begin{align}\label{b2}
  N_m(\cTbn) = \sum_{j=1}^{|\mathbf{n}|}\indic{(d_{j+i-1})_{i=1}^m\in\bbD_m}
,\end{align}
where again $\bd=(d_i)_{i=1}^{|\bn|}$ is a uniformly random sequence in $\bbB_\bn$
constructed by \eqref{ctd2}.
Hence, using also the rotational symmetry, 
\begin{align}\label{b3}
\E  N_m(\cTbn) &= |\bn| \P\bigsqpar{(d_{i})_{i=1}^m\in\bbD_m}
.\end{align}
By \eqref{ctd2}, the  subsequence $(d_1,\dots,d_m)$ is obtained by drawing
without replacement from some fixed multiset $\set{d_1',\dots,d_{|\bn|}'}$, and
thus its distribution is exchangeable. 
Since each sequence in $\bbD_m$ corresponds to $m$ sequences in $\bbB_m$
(its cyclic shifts),
and all these have the same probability to appear,
it follows from \eqref{b3} and \eqref{Bn} that
\begin{align}\label{b3b}
\E  N_m(\cTbn) &
= \frac{|\bn|}{m}\P\bigsqpar{(d_{i})_{i=1}^m\in\bbB_m}
= \frac{|\bn|}{m}\P\lrsqpar{\sum_{i=1}^md_{i}=m-1}
.\end{align}

More generally, for any integers $m\ge1$ and $r\ge1$ with $rm\le |\bn|$, 
$(N_m(\cTbn))_r$ is the number of $r$-tuples of distinct subsequences 
of $\bd$ (again regarded cyclically) that belong to $\bbD_m$. 
These $r$ subsequences
have to be disjoint. The position of the first can be chosen in $|\bn|$ ways,
and then we have to select $r-1$ disjoint 
subsequences of length $m$ from an interval
of length $|\bn|-m$; this can be done in $\bigpar{|\bn|-m-(r-1)(m-1)}_{r-1}$
ways.
Again, by symmetry, having chosen the positions of these subsequences, the
probability that all are degree sequences in $\bbD_m$ does not depend on the
positions;
hence we obtain, reducing to $\bbB_m$ as in \eqref{b3b}, for $|\bn|\ge rm$,
\begin{align}\label{b4}
&\E  \bigpar{N_m(\cTbn)}_r 
= |\bn|\xpar{|\bn|-rm+m-1}_{r-1}
\P\lrsqpar{(d_{i})_{i=(j-1)m+1}^{jm}\in\bbD_m\text{ for }j=1,\dots,r}
\notag\\&\qquad
= \frac{|\bn|\xpar{|\bn|-rm+m-1}_{r-1}}{m^{r}}
\P\lrsqpar{(d_{i})_{i=(j-1)m+1}^{jm}\in\bbB_m\text{ for }j=1,\dots,r}
\notag\\&\qquad
= \frac{|\bn|\xpar{|\bn|-rm+m-1}_{r-1}}{m^{r}}
\P\lrsqpar{\sum_{i=(j-1)m+1}^{jm}(d_i-1)=-1\text{ for }j=1,\dots,r}
.\end{align}

\begin{lemma}\label{LM1}
  Assume \refConds{Condition1} and \ref{Cond2}, and that the limit
  distribution $\bp$ is nonlattice.
Let $m_\kk$, $\kk\ge1$, be integers with $1\le m_\kk\le|\bnk|$
such that $m_\kk\to\infty$ and  $|\bnk|-m_\kk\to\infty$
as \kktoo.
Then, as \kktoo,
\begin{align}\label{lm1}
\E [N_{m_\kk}(\cTbnk)] 
\sim \frac{|\bnk|}
  {\sqrt{2\pi}\gsp (1-\frac{m_\kk}{|\bnk|})\strut\qq m_\kk^{3/2}}
.\end{align}
Consequently,
\begin{romenumerate}
\item \label{LM1<}
If\/ $m_\kk\ll |\bnk|^{2/3}$, then $\E [N_{m_\kk}(\cTbnk)] \to\infty$.
\item \label{LM1=}
If\/ $m_\kk\sim a |\bnk|^{2/3}$ for some constant $0<a<\infty$, 
then $\E [N_{m_\kk}(\cTbnk)] \to (2\pi\gssp a^3)\qqw$.
\item \label{LM1>}
If\/ $m_\kk\gg |\bnk|^{2/3}$, then $\E [N_{m_\kk}(\cTbnk)] \to0$.
\end{romenumerate}
\end{lemma}

\begin{proof}
Note that the assumption that $\bp$ is nonlattice entails that 
$\gss_\bp>0$.
  We apply \refT{TLLT}, choosing any fixed sequences
  $\bd_\kk=(d_{i\kk})_{i=1}^{\xbnk}\in\bbB_{\bnk}$.
(The notation in \refApp{ALLT} differs slightly from the one used here,
since in \refApp{ALLT} we let $\bd$ be deterministic, and randomize in
\eqref{Sm}; 
thus $\bd$ there corresponds to $(d'_i)_i$ in \eqref{ctd2}.
However, the random sums in \eqref{b3b} and \eqref{Sm} have the
same distribution.)
Then, the variables defined in \eqref{apa}--\eqref{apc} become,
using the random variables $D_\kk$ defined in \refR{RD}
together with \eqref{the}, \eqref{cond2b}, and \eqref{ape}, as $\kk \to  \infty$,
\begin{align}\label{lm1d}
  \bxd_\kk&=\E  D_\kk = 1-|\bnk|\qw,
\\\label{lm1Q}
Q_\kk&=\xbnk\Var D_\kk = |\bnk|\gssk\sim |\bnk|\gssp,
\\\label{lm1gs}
\hgssk&=m_\kk\Bigpar{1-\frac{m_\kk}{|\bnk|}}\gssk
\sim m_\kk\Bigpar{1-\frac{m_\kk}{|\bnk|}}\gssp
.\end{align}
In particular, \eqref{lm1Q} shows that
\refCond{Cond1Z} follows from \refConds{Condition1} and \ref{Cond2}.
Note that the assumptions $m_\kk\to\infty$ and $|\bnk|-m_\kk\to\infty$
imply $\hgssk\to\infty$. (Use \eqref{lm1gs} and consider the cases 
$m_\kk\le\frac12|\bnk|$ and $m_\kk>\frac12|\bnk|$ separately.)
Moreover,
by \refR{RD}, the random variables $D_\kk^2$ are \ui, and thus so are
$|\Dk-\E\Dk|^2$. In other words,
see \eg{} \cite[Definition 5.4.1]{Gut},
\begin{align}\label{lm1e}
  \lim_{a\to\infty} \sup_{\kk}\frac{1}{\xbnk}
\sum_{|d_{i\kk}-\bxd_\kk|>a}|d_{i\kk}-\bxd_\kk|^2 =0
\end{align}
Since, as just said, $\hgssk\to\infty$, it follows from \eqref{lm1e} 
and \eqref{lm1Q}
that the Lindeberg condition \eqref{Lindeberg} holds.

Consequently, \refT{TLLT} applies.
We choose there $k:=m_\kk-1$; then $|k-m_\kk\bxd_\kk|\le1$, 
and thus \eqref{tllt} yields,
using \eqref{lm1gs},
\begin{align}\label{lm1f}
\P\lrsqpar{\sum_{i=1}^{\mk}d_{i\kk}=\mk-1}
=
\P\bigsqpar{\Smk=m_\kk-1}
\sim 
\frac{1}
 {\sqrt{2\pi \hgssk} }
=\frac{1}
 {\sqrt{2\pi\gssk \mk} (1-\frac{m_\kk}{\xbnk})\strut\qq }
.\end{align}
Hence, \eqref{b3b} yields \eqref{lm1}.

The conclusions \ref{LM1<}--\ref{LM1>} are immediate consequences of
\eqref{lm1};
for \ref{LM1>}, we again
consider the cases 
$m_\kk\le\frac12|\bnk|$ and $m_\kk>\frac12|\bnk|$ separately.
\end{proof}

\begin{theorem}\label{TS}
  Assume \refConds{Condition1} and \ref{Cond2}, and that the limit
  distribution $\bp$ is nonlattice.
Let $m_\kk$, $\kk\ge1$, be integers with $1\le m_\kk\le|\bnk|$
such that $m_\kk\sim a\xbnk^{2/3}$ as \kktoo,
for some constant $a>0$.  
Then, as \kktoo,
\begin{align}\label{ts}
N_{m_\kk}(\cTbnk) \dto\Po\bigpar{(2\pi\gssp a^3)\qqw}
.\end{align}
\end{theorem}

\begin{proof}
  We use \eqref{b4} and the method of moments.
Let $\gl:=(2\pi\gssp a^3)\qqw$; we have already shown in \refL{LM1} that 
$\E N_{\mk}(\cTbnk) \to\gl$ as \ktoo.
Now fix $r\ge2$. Then \eqref{b4} yields
\begin{align}\label{ts1}
  \E  \bigpar{N_{\mk}(\cTbnk)}_r &
\sim \lrpar{\frac{\xbnk}{\mk}}^r
\P\lrsqpar{\sum_{i=(j-1)m_\kk+1}^{jm_\kk}(d_{i\kk}-1)=-1\text{ for }j=1,\dots,r}
,
\end{align}
where $\bd_\kk=(d_{i\kk})_{i=1}^{\xbnk}$ is a uniformly random sequence in $\bbB_\bnk$.

Consider for simplicity first the case $r=2$.
Condition on any values of $d_{1\kk},\dots,d_{\mk\kk}$ such that 
\begin{align}
  \label{ts2}
\sum_{i=1}^{\mk}(d_{i\kk}-1)=-1. 
\end{align}
There remains $n_\kk'(i):=n_\kk(i)-\numset{j\le \mk:d_{j\kk}=i}$ degrees $i$
in our degree sequence $\bd_\kk$, for each $i\ge0$, and the
values $d_{\mk+1,\kk},\dots,d_{2\mk,\kk}$ are obtained by drawing without
replacement from this remaining sequence.
The sequence $\bn_\kk':=(n_\kk'(i))_{i\ge0}$ is not quite a degree statistic
according to our definitions, since we have
\begin{align}\label{ts3}
  \sum_{i\ge0}(i-1)n_\kk'(i) 
=
  \sum_{i\ge0}(i-1)n_\kk(i) -\sum_{j=1}^{\mk} (d_{j,\kk}-1) 
=0
\end{align}
instead of $-1$ as in \eqref{ds}. However, this is of no importance 
in the asymptotic calculations in \eqref{lm1f} based on \refT{TLLT}. 
Moreover, we have $n_\kk'(i)\le n_\kk(i)$ and $n_\kk'(i)=n_\kk(i)-O(\xbnk^{2/3})$
for all $i \ge 0$, and $|\bn_\kk'|=|\bnk|-\mk\sim\xbnk$.
It follows that, 
uniformly for any choices of $d_{i\kk}$ ($1\le i\le\mk$ and 
$\kk\ge1$) satisfying \eqref{ts2}, the sequences $\bn_\kk'$ satisfy
\refCond{Condition1} 
(except for \eqref{ds} as just said), with the same $\bp$. 
Moreover, \refCond{Cond2} too holds for $\bn_\kk'$.
To see this, we use the equivalent formulation with
uniform integrability of $\Dk^2$ in \refR{RD2}.
We can construct the random degree $\Dk$ as
$\Dk:=d_{I_\kk,\kk}$ where
$I_k$ is a uniformly random index in $[|\bnk|]$; furthermore, since the
order of $d_{1\kk},\dots,d_{|\bnk|,\kk}$ does not matter here, we can
as well first
condition on any given $d_{1\kk},\dots,d_{\mk,\kk}$.
Then,
the corresponding random degree defined by
$\bn_\kk'$ is $\Dk':=(\Dk\mid I_\kk>m_\kk)$, i.e., $\Dk$ conditioned on
$I_\kk>m_\kk$. Since 
the random variables $\Dk^2$ are uniformly integrable and
$\P(I_\kk>m_\kk)=1-m_\kk/|\bnk|\to1$,
it follows 
(from the definition of uniform integrability \cite[Section 5.4]{Gut})
that the variables $(\Dk')^2=(\Dk^2\mid I_\kk>m_\kk)$ also are uniformly
integrable. Thus, by \refR{RD2}, 
\refCond{Cond2} holds also for the sequences $\bn_\kk'$,
for any choices of drawn values $d_{1\kk},\dots,d_{\mk\kk}$
(and therefore uniformly in all such choices).
Hence, the argument yielding \eqref{lm1f} yields also,
with $\Smk':=\sum_{i=\mk+1}^{2\mk}d_{i\kk}$ and
noting that $\mk/\xbnk\to0$,
\begin{align}\label{ts4}
\P\lrsqpar{\sum_{i=\mk+1}^{2\mk}d_{i\kk}=\mk-1\Bigm| d_{1\kk},\dots,d_{\mk\kk}}
=
\P\bigsqpar{\Smk'=m_\kk-1}
\sim 
\frac{1} {\sqrt{2\pi\gssk \mk}}
,\end{align}
uniformly for all $d_{1\kk},\dots,d_{\mk\kk}$ such that
$\sum_{i=1}^{\mk}d_{i\kk}=\mk-1$. Thus
\begin{align}\label{ts52}
&\P\lrsqpar{\sum_{i=1}^{\mk}d_{i\kk}=\sum_{i=\mk+1}^{2\mk}d_{i\kk}=\mk-1}
                           \notag\\&\qquad
=
\E\lrsqpar{
\P\lrsqpar{\sum_{i=\mk+1}^{2\mk}d_{i\kk}=\mk-1\Bigm|  d_{1\kk},\dots,d_{\mk\kk}}
\lrindic{
\sum_{i=1}^{\mk}d_{i\kk}=\mk-1}}
\notag\\&\qquad
\sim 
\frac{1} {\sqrt{2\pi\gssk \mk}}
\P\lrsqpar{\sum_{i=1}^{\mk}d_{i\kk}=\mk-1}
\sim 
\lrpar{\frac{1} {\sqrt{2\pi\gssk \mk}}}^2
.\end{align}

For larger $r$, we argue in the same way, now conditioning on
$d_{i\kk}$ for $1\le i\le (r-1)\mk$, and we obtain by induction that for
every fixed $r\ge1$,
\begin{align}\label{ts5}
  \P\lrsqpar{\sum_{i=(j-1)m+1}^{jm}(d_{i\kk}-1)=-1\text{ for }j=1,\dots,r}
\sim 
\lrpar{\frac{1} {\sqrt{2\pi\gssk \mk}}}^r
.\end{align}
Consequently, \eqref{ts1} yields
\begin{align}\label{ts6}
  \E  \bigpar{N_{\mk}(\cTbnk)}_r &
\sim 
\lrpar{\frac{\xbnk} {\sqrt{2\pi\gssk \mk^3}}}^r
\sim \lrpar{\frac{1} {\sqrt{2\pi\gssk a^3}}}^r
= \gl^r.
\end{align}
Hence, $N_{\mk}(\cTbnk) \dto\Po(\gl)$ by the method of moments.
\end{proof}

\begin{remark}\label{RTS}
The assumptions
$m_\kk\to\infty$ and  $|\bnk|-m_\kk\to\infty$ in \refL{LM1}
are necessary, in general.

First, if $m_\kk=O(1)$, we may by considering subsequences assume that
$m_\kk=m$ is constant, and then $N_{m}(\cTbnk)$ is a sum of a finite number
of subtree counts $N_{T_j}(\cTbnk)$ with fixed $T_j$. 
(The trees $T_j$ are all the trees in $\bbT_{m}$.)
This is treated in \cite{SJ378}. 
For every large $m$, we then have $\E N_{m}(\cTbnk) \sim c(m)|\bn_\kk|$
with $c(m)>0$,
but there might be some small $m$ for which 
$\E N_{m}(\cTbnk)$ is smaller, or even vanishes.
For an example, suppose that $n_\kk(i)>0$ only for $i\in\set{0,9,10}$ and
$p_9=p_{10}=1/19$ (to get the correct mean). 
Then $\bp$ is nonlattice, but $N_m(\cTbnk)=0$ for $2\le m\le 9$.

Similarly, if $\mk=\xbnk-\ell$ for some fixed $\ell$, it follows from
\eqref{b3b} and symmetry that
\begin{align}\label{gwi}
\E  N_{\xbnk-\ell}(\cTbnk) &
\sim\P\lrsqpar{\sum_{i=1}^{\xbnk-\ell}d_{i}=\xbnk-\ell-1}
=\P\lrsqpar{\sum_{j=1}^{\ell}d_{j}=\ell}
.\end{align}
For $\ell=0$, this equals 1. (There is trivially always exactly one fringe
subtree of size $\xbnk$, viz.\ the entire tree $\cTbnk$.)
For $\ell\ge1$, \eqref{gwi} converges  to some number $\pi_\ell\in[0,1)$,
and then
the asymptotic distribution of $N_{\xbnk-\ell}(\cTbnk)$ is $\Be(\pi_\ell)$;
note that there is not room for more than one fringe tree of this size when
$\xbnk>2\ell$. We have $\pi_\ell>0$ for all large $\ell$, but the example
above shows that $\pi_\ell=0$ is possible for some small $\ell$.
\end{remark}

\begin{remark}
  We conjecture that when $\mk\ll\xbnk^{2/3}$, and thus
  $\E[N_\mk(\cT_\bnk)]\to\infty$ by \refL{LM1},
we have asymptotic normality, in analogy with \refTs{The01}, \ref{The1}, and
\ref{ThPo2}. It might be possible to prove this using the Gao--Wormald theorem,
as in the other theorems just mentioned, but since this would require
estimates of moments of unbounded order, the estimates in the proof of
\refT{TS} are not precise enough; we would need either a better estimate of
the error in \refT{TLLT}, or a different, non-inductive, way to estimate
\eqref{b4} or \eqref{ts1}. 
\end{remark}

\begin{remark}
  We have assumed in both \refL{LM1} and \refT{TS} that
the asymptotic degree distribution $\bp$ is nonlattice.
The results are easily extended to the case when all degrees $d_{i\kk}$
are divisible
by some fixed integer $h>1$, for example if all degrees are even; we may
then apply  \refT{TLLT} to the integers $d_{i\kk}/h$.
There are, however, also limiting cases.
For example, suppose that 
$\bp$ is concentrated on the even integers, and thus lattice, but that
$n_\kk(1)>0$ although
$n_\kk(1)=o(\xbnk)$.
In \refE{Eperiodic}  below, one example of this type is studied in some
detail,
and it is seen that
in the critical case $m_\kk=\Theta(|\bnk|^{2/3})$,
there will periodicities in the result if 
$n_\kk(1)=O(\xbnk^{1/3})$, but it seems that these disappear if $n_\kk(1)$
is larger.
We believe that this behaviour is typical for cases with  $\bp$ lattice,
but 
we have not pursued this further, and leave such cases to the reader.
\end{remark}

\begin{example}\label{Eperiodic}
Suppose that 
$\xbnk\to\infty$ with
$n_\kk(i)>0$ only for $i\in\set{0,1,2}$, and that 
$n_\kk(1)\sim b\xbnk^{1/3}$ 
for some constant $b>0$.
Then it follows from \eqref{ds} that necessarily
$n_\kk(0)=\frac12(|\bnk|-n_\kk(1)+1)=n_\kk(2)+1$, and thus 
\refCond{Condition1} holds with $p_0=p_2=\frac12$;
hence $\bp$ is concentrated on $\set{0,2}$ and thus has span 2.
\refCond{Cond2} holds too.

Let $m_\kk\sim a\xbnk^{2/3}$ for some $a>0$, let $\bd$ be uniformly random
in $\bbB_{\bn_{\kk}}$, and let
$Y_\kk:=|\set{i\le m_\kk:d_{i\kk}=1}|$. 
It is easy to see that as $\kk\to\infty$, $Y_\kk\dto Y\simx\Po(ab)$.
Conditioned on $Y_\kk$, the remaining $\mk-Y_\kk$ degrees $d_{i\kk}$ with
$i\le\mk$ are all 0 or 2, so their sum is even, and the number of $2$s has a
hypergeometric distribution; it follows easily from this 
(for example using \refT{TLLT}) that for every fixed $y\in\bbNo$, 
\begin{align}
  \P[S_\mk=\mk-1\mid Y_\kk=y]
=
  \begin{cases}
    0,& y\not\equiv\mk-1\pmod2,
\\
\frac{2}{\sqrt{2\pi\mk}}(1+o(1)),&y\equiv\mk-1\pmod2.
  \end{cases}
\end{align}
Furthermore, this holds uniformly for all $y\le m_\kk\qq$, say, and it follows
easily from \eqref{b3b} that, cf.\ \eqref{lm1f} and \refL{LM1}\ref{LM1=}
in the nonlattice case,
\begin{align}
\E N_\mk(\cT_\bnk)
&=
\frac{\xbnk}{\mk}  \P[S_\mk=\mk-1\mid Y_\kk=y]
\sim\frac{2\xbnk}{\sqrt{2\pi\mk^3}}\P\bigpar{Y_\kk\equiv \mk-1\pmod2}
\notag\\&
\sim
  \begin{cases}
    \frac{1+e^{-ab}}{\sqrt{2\pi a^3}},& \mk\text{ is odd},
\\
    \frac{1-e^{-ab}}{\sqrt{2\pi a^3}},& \mk\text{ is even}.
  \end{cases}
\end{align}
Higher factorial moments can be calculated asymptotically by the method in
the proof of \refT{TS}, and it follows that 
$N_\mk(\cT_\bnk)$ has one asymptotic Poisson distribution for even $\mk$,
and a different asymptotic Poisson distribution for odd $\mk$.
\end{example}

\section{Application to Galton--Watson Trees} \label{AppGW}

As in our previous paper \cite{SJ378} for the number of fringe trees equal
to a fixed tree $T$ (not depending on $\kk$), 
we can use results for random trees with given vertex degrees to obtain
corresponding results also for conditioned \GWt{s} (or the somewhat more
general simply generated trees) by conditioning on the degree statistic of
the tree. For the results in the present paper, 
where we consider fringe trees $T_\kk$ growing with $\kk$, 
such results for \GWt{s} have been shown 
by Cai and Devroye \cite{Cai2017} by other methods.
(See also \cite{SJ285} for the case of a fixed tree $T_\kk=T$.)
We will here illustrate the conditioning method by giving
alternative proofs of two of their results, in one case extending it to
include also some 
offspring distributions with infinite variance.

We consider for simplicity here a \cGWt{} $\cT_{\bp,n}$ obtained from a
\GWt{} $\cT_\bp$ with offspring distribution $\bp$ having mean 1 by
conditioning on the size $|\cT_\bp|=n$; it is well known, using tiltings, 
that the results immediately extend to \cGWt{s} with a large class of offspring
distributions, and even further to a large class of simply generated trees,
see e.g.\ \cite{SJ264}. 

The key fact is (as in the related arguments in \cite{SJ378})  
that
for any fixed degree statistic $\mathbf{n}=(n(i))_{i \geq 0}$ with
$\mathbb{P}(\mathbf{n}_{\mathcal{T}_{\mathbf{p},n}} = \mathbf{n}) >0$,
it follows from \eqref{pip},
see e.g.\ \cite[Proposition 8]{Addario2022},
  that 
\begin{align}\label{gw0}
  \text{conditionally given 
$\mathbf{n}_{\mathcal{T}_{\mathbf{p},n}} =\mathbf{n}$, 
we have $\mathcal{T}_{\mathbf{p},n} \simx {\rm Unif}(\mathbb{T}_{\mathbf{n}})$}
.\end{align}
 
We begin with a straightforward conditioning argument applied to the number
of fringe trees of a given size. The result is essentially a special case of
\cite[Theorem 1.4(ii)]{Cai2017} 
(there proved directly by the method of moments),  
although we have also computed the (asymptotic) expectation.

\begin{theorem}[{\cite[Theorem 1.4(ii)]{Cai2017}}]\label{TheoGW2}
Let $\mathbf{p} := (p_{i})_{i \geq 0}$ be a probability distribution on
$\bbNo$ such that $p_{0} >0$ and $p_{i} >0$ for at least one $i \geq 2$. Let
$\mathcal{T}_{\mathbf{p},n}$ be a Galton--Watson tree 
with offspring distribution $\mathbf{p}$ conditioned to have
size $n \in \mathbb{N}$
(we only consider $n$ along values for which the conditioning is well
defined). 
Furthermore, suppose that $\mathbf{p}$ is nonlattice, 
the mean $\sum_{i  \geq 0} i p_{i} = 1$, and the variance
$\sigma^{2}_{\mathbf{p}} \coloneqq \sum_{i \geq  0}(i-1)^{2}p_{i} <\infty$. 
Let $m_n$ be integers with $1\le m_n \le
n$ 
such that $m_n \sim a n^{2/3}$ as $n \rightarrow \infty$,
for some constant $a>0$.  
Then, as $n \rightarrow \infty$,
\begin{align} \label{eq15GW}
N_{m_n}(\mathcal{T}_{\mathbf{p},n}) \dto\Po\bigpar{(2\pi\gssp a^3)\qqw}
.\end{align}
\end{theorem}

\begin{proof}
Let $\mathcal{P}(\mathbb{N}_{0})$ be the set of probability measures on $\mathbb{N}_{0}$ equipped with the weak topology. By \cite[Lemma 11]{Broutin2014}, we have that,  as $n \rightarrow \infty$,
\begin{align} \label{eq16GW}
\Big(\mathbf{p}(\mathbf{n}_{\mathcal{T}_{\mathbf{p},n}}), \sigma_{\mathbf{p}(\mathbf{n}_{\mathcal{T}_{\mathbf{p},n}})}^{2} \Big) \dto (\mathbf{p}, \sigma_{\mathbf{p}}^{2}), 
\end{align}
\noindent where the convergence holds in the space $\mathcal{P}(\mathbb{N}_{0}) \times \mathbb{R}$ equipped with the product topology.
By the Skorohod coupling theorem \cite[Theorem
4.30]{Kallenberg2002}, we can and will assume that the convergence in
\eqref{eq16GW} holds a.s.; in other words, \refCond{Condition1} and \refCond{Cond2}  hold a.s.\ 
for the degree statistics $\mathbf{n}_{\mathcal{T}_{\mathbf{p},n}}$, with $\bp$. Moreover, e.g.\ by resampling $\mathcal{T}_{\mathbf{p},n}$ conditioned on 
$\mathbf{n}_{\mathcal{T}_{\mathbf{p},n}}$, we may assume that 
conditioned on the sequence of degree statistics
$(\mathbf{n}_{\mathcal{T}_{\mathbf{p},n}})_{n=1}^\infty$,
the random trees $\cT_{\bp,n}$, $n\ge1$,  are independent. 

As noted above,
for any fixed degree statistic $\mathbf{n}$ with
$\mathbb{P}(\mathbf{n}_{\mathcal{T}_{\mathbf{p},n}} = \mathbf{n}) >0$,
we have \eqref{gw0}.
It follows that we may apply Theorem \ref{TS}
conditioned on the 
sequence of degree statistics
$(\mathbf{n}_{\mathcal{T}_{\mathbf{p},n}})_{n=1}^\infty$; 
this shows that (\ref{eq15GW}) holds
conditioned on
$(\mathbf{n}_{\mathcal{T}_{\mathbf{p},n}})_{n=1}^\infty$. Then, (\ref{eq15GW})
also holds unconditionally by the dominated convergence theorem. 
\end{proof}

The same argument can be used to give versions for the \cGWt{} $\cT_{\bp,n}$
of the Poisson convergence parts of \refTs{The01} and \ref{ThPo2}.
In principle, it should be possible to apply this method also to the parts
with normal convergence, but then we would have to also take into account the
random fluctuations of the degree statistic of the \cGWt{} $\cT_{\bp,n}$,
see \cite[Section 7]{SJ378} where this is done in detail for 
$N_T(\cT_{\bp,n})$ where $T$
is a fixed tree. Instead we go back to the proofs above and use moment calculations there
to show the following result
which includes \cite[Theorem 1.3]{Cai2017}
(although without an explicit error bound)
and extends it to 
some cases where $\bp$ has infinite variance.

\begin{theorem}[Partly {\cite[Theorem 1.3]{Cai2017}}] \label{TheoGW1}
Let $\mathbf{p} := (p_{i})_{i \geq 0}$ be a probability distribution on
$\bbNo$ such that $p_{0} >0$ and $p_{i} >0$ for at least one $i \geq 2$. Let
$\mathcal{T}_{\mathbf{p},n}$ be a Galton--Watson tree 
with offspring distribution $\mathbf{p}$ conditioned to have
size $n \in \mathbb{N}$
(we only consider $n$ along values for which the conditioning is well
defined). 
Furthermore, suppose that the mean 
$\sum_{i \geq 0} i p_{i} = 1$ and that
$\mathbf{p}$ satisfies one of the following conditions: 
\begin{enumerate}[label=\upshape(\alph*)]
\item\label{case1GW}  
$\sigma^{2}_{\mathbf{p}} 
\coloneqq \sum_{i \geq 0}(i-1)^{2}p_{i} \in  (0,\infty)$. 
(I.e., the distribution $\bp$ has finite variance $\gss$.) 
\item $\sigma^{2}_{\mathbf{p}} = \infty$ and
  $\mathbf{p}$ belongs to the domain of attraction of a stable law
  of index $\alpha \in (1,2]$. 
(The last condition is equivalent to that
there exists a slowly varying function
  $L:\mathbb{R}_{+} \rightarrow \mathbb{R}_{+}$ such that $\sum_{i = 0}^{k}
  i^{2} p_{i} = k^{2-\alpha} L(k)$,
  as $k \rightarrow \infty$
\cite[Theorem XVII.5.2]{FellerII}.) \label{case2GW}
\end{enumerate}
\noindent Let $T_{n} \in \mathbb{T}$ be such that $m_{n} \coloneqq |T_{n}| \rightarrow \infty$. As $n \rightarrow \infty$, 
\begin{enumerate}[label=\upshape(\roman*)]
\item \label{ReGW1} 
If\/ $n \pi_{\mathbf{p}}(T_{n}) \rightarrow \lambda$, 
for some $\lambda \in [0,\infty)$, then 
\begin{align}  \label{ReGW1Conv}
N_{T_n}(\mathcal{T}_{\mathbf{p},n}) \dto\Po\bigpar{\gl}
\end{align} 
\noindent with convergence of all moments.
\item \label{ReGW2} If\/ $n \pi_{\mathbf{p}}(T_{n}) \rightarrow \infty$, then 
\begin{align} \label{ReGW2Conv}
\frac{N_{T_{n}}(\mathcal{T}_{\mathbf{p},n}) - n \pi_{\mathbf{p}}(T_{n})}{\sqrt{n \pi_{\mathbf{p}}(T_{n})}} \dto\N\bigpar{0,1}.
\end{align}
\end{enumerate}
\end{theorem}

\begin{proof}
We use again \eqref{gw0}.
Hence, it follows from \cite[Lemma 3.3(i)]{SJ378},
see \eqref{eq4G}, that, for $q \in \mathbb{N}$ such that $n \geq qm_{n}-q+1$,
\begin{align} \label{eq1GW}
\E (N_{T_{n}}(\mathcal{T}_{\mathbf{p},n}))_{q} =  \frac{n}{(n)_{qm_{n}-q+1}} \mathbb{E} \Big[ \prod_{i\geq 0} (n_{\mathcal{T}_{\mathbf{p},n}}(i))_{qn_{T_{n}}(i)} \Big].
\end{align}

Let $\xi_{1}, \xi_{2}, \dots$, be a sequence of independent random variables with
distribution $\mathbf{p}$, and define the partial sums $S_{n} = \sum_{i
  =1}^{n} \xi_{i}$, $n \geq 0$. 
It follows from \eqref{eq1GW} together with the bijection 
$\gU:\bbT_n\leftrightarrow\bbD_n$ and 
the $n$-to-$1$ map $\Phi:\bbB_n\to\bbD_n$ discussed
in \refS{TreesandDegrees} that 
\begin{align}  \label{eq2GW}
\E (N_{T_{n}}(\mathcal{T}_{\mathbf{p},n}))_{q} =  n \prod_{i\geq 0} p_{i}^{qn_{T_{n}}(i)} \cdot \frac{(n)_{q m_{n}}}{(n)_{qm_{n}-q+1}}  \frac{\mathbb{P}(S_{n-qm_{n}} = n-1-q(m_{n}-1))}{\mathbb{P}(S_{n} = n-1)};
\end{align}
see e.g.\ \cite[(B.4)]{SJ378} (with $q_{in} = q n_{T_{n}}(i)$, for $i \geq 0$). 

We first prove \ref{ReGW1} for $\gl>0$ and \ref{ReGW2}. 
Since then $n \pi_{\mathbf{p}}(T_{n}) \rightarrow \lambda$, as $n
\rightarrow \infty$, for some $\lambda \in (0,\infty]$, we may consider $n$ large enough such that $\pi_{\mathbf{p}}(T_{n}) >0$. Let $q_{n} \in \mathbb{N}$ be such that 
\begin{align}\label{lc1}
q_{n}=O((n \pi_{\mathbf{p}}(T_{n}))^{1/2}). 
\end{align}
Note that then
\begin{align} \label{eq3GW}
\frac{q^{2}_{n}m_{n}^{2}}{n} = \frac{q^{2}_{n} m_{n}^{2}\pi_{\mathbf{p}}(T_{n})}{n\pi_{\mathbf{p}}(T_{n})} = O\Big(m_{n}^{2} \pi_{\mathbf{p}}(T_{n}) \Big) = O\Big(m_{n}^{2} \sup_{i \geq 0} p_{i}^{m_{n}}\Big) = o(1). 
\end{align}
\noindent This allows us to consider only $n$ large enough such that $n \geq
q_{n}m_{n}-q_{n}+1$.  Then, by \eqref{eq2GW},  \eqref{eq5Gb}, and \eqref{eq3GW},
\begin{align} \label{eq4GW}
\E (N_{T_{n}}(\mathcal{T}_{\mathbf{p},n}))_{q_{n}} & =  (n\pi_{\mathbf{p}}(T_{n}))^{q_{n}}  \frac{\mathbb{P}(S_{n-q_{n}m_{n}} = n-1-q_{n}(m_{n}-1))}{\mathbb{P}(S_{n} = n-1)}  \cdot \exp ( o(1)).
\end{align}
\noindent Next, we use a suitable local limit theorem for the sums $(S_n, n \geq 0)$. We consider the cases \ref{case1GW} and \ref{case2GW} separately.

\pfitemref{case1GW}
In this case, the distribution $\mathbf{p}=(p_{i})_{i \geq 0}$ has mean $1$ and finite variance $\sigma_{\mathbf{p}}^{2} >0$. Let $h \geq  1$ be the span of this distribution, i.e., the largest integer such that $\xi_{1}$ a.s. is a multiple of $h$. The standard local central limit theorem, see e.g.\ \cite[Theorem 7.1]{Petrov}, yields, recalling \eqref{eq3GW}, 
\begin{align} 
\mathbb{P}(S_{n-q_{n}m_{n}} = n-1-q_{n}(m_{n}-1)) & 
= \frac{h}{\sqrt{2 \pi \sigma_{\mathbf{p}}^{2} n}} 
\exp \Big(-\frac{(q_{n}-1)^2}{2\sigma^{2}_{\mathbf{p}}(n-q_nm_n)} +o(1) \Big) 
\nonumber \\
& = \frac{h}{\sqrt{2 \pi \sigma_{\mathbf{p}}^{2} n}} \exp(o(1)), \label{eq5GW}  \\ 
\mathbb{P}(S_{n} = n-1) & = \frac{h}{\sqrt{2 \pi \sigma_{\mathbf{p}}^{2} n}} \exp(o(1)). \label{eq6GW}
\end{align}
\noindent Hence, \eqref{eq4GW} yields,
\begin{align} 
\E (N_{T_{n}}(\mathcal{T}_{\mathbf{p},n}))_{q_{n}}  =  (n\pi_{\mathbf{p}}(T_{n}))^{q_{n}} \cdot  \exp( o(1)). \label{eq7GW}
\end{align}

In case \ref{ReGW1},  
the condition \eqref{lc1} says $q=O(1)$.
Hence, for every fixed $q \in \mathbb{N}$, \eqref{eq7GW} applies and yields
\begin{align} \label{eq8GW}
\E (N_{T_{n}}(\mathcal{T}_{\mathbf{p},n}))_{q} \sim (n\pi_{\mathbf{p}}(T_{n}))^{q},
\end{align}
\noindent as $n \rightarrow \infty$; 
consequently, \eqref{ReGW1Conv} follows by the methods of moments. Note also
that \eqref{eq8GW} implies the convergence of all moments of
$N_{T_{n}}(\mathcal{T}_{\mathbf{p},n})$ as claimed in \ref{ReGW1}. 

In case \ref{ReGW2}, it follows from \eqref{eq7GW} that, for $q_{n}=O((n \pi_{\mathbf{p}}(T_{n}))^{1/2})$, 
\begin{align}  \label{eq9GW}
\E (N_{T_{n}}(\mathcal{T}_{\mathbf{p},n}))_{q_{n}} =  (n\pi_{\mathbf{p}}(T_{n}))^{q_{n}} \cdot \exp \Big(\frac{n\pi_{\mathbf{p}}(T_{n})  - n\pi_{\mathbf{p}}(T_{n})}{2 n^{2}\pi_{\mathbf{p}}(T_{n})^{2}}q^{2}_{n} + o(1)\Big).
\end{align}
\noindent Hence, \eqref{ReGW2Conv} follows by applying the Gao--Wormald theorem \cite[Theorem 1]{GW2004} (or \cite[Theorem
A.1]{SJ378} with $m=1$) with $\mu_{n} = \sigma_{n}^{2}  = n\pi_{\mathbf{p}}(T_{n})$. 

\pfitemref{case2GW}
This is similar. 
By assumption, there exist sequences of constants $a_n>0$
and $b_n$ such that $(S_n-b_n)/a_n$ converges in distribution to some stable
random variable $Y$ of index $\alpha \in (1,2]$; since $\ga>1$, we may here
choose $b_n=\E S_n=n$
\cite[Theorem XVII.5.3]{FellerII}.
Moreover, it follows from \cite[XVII.(5.24)]{FellerII} that the sequence
$a_n$ is regularly varying with exponent $1/\ga$, and in particular that
that if $n'(n)=n+o(n)$, then $a_{n'(n)}\sim a_n$.

Let $g_Y(x)$ be the
density function of the stable limit $Y$, and note that $g_Y(x)$ is
continuous and that
$g_Y(x)>0$ for every $x\in\bbR$
(see e.g.\ \cite[Remark 4 after Theorem 2.2.3, pp.\ 79--80]{Zolotarev1986}).
Since we assume that the variance of $\xi_1$ is
infinite, we have $a_n\gg n\qq$ (see \cite[XVII.(5.23)]{FellerII}). 
Let again $h$ be the span of the distribution $\mathbf{p}$, and note that
$\cT_{\bp,n}$ exists only if $n\equiv1\pmod h$, and that $\pi_\bp(T_n)>0$
implies $m_n=|T_n|\equiv1\pmod h$; we may thus assume these congruences.
Then, by
recalling \eqref{eq3GW}, the local limit theorem 
(see \cite[\S\ 50]{GneKol}) yields,
with $n'(n):=n-q_nm_n=n-o(n)$,  
\begin{align}
\mathbb{P}(S_{n-q_{n}m_{n}} = n-1-q_{n}(m_{n}-1)) 
& = \frac{h}{a_{n'(n)}} \Bigpar{g_{Y}\Bigpar{\frac{q_n-1}{a_{n'(n)}}}+o(1)} 
= \frac{h}{a_{n}} g_{Y}(0) \exp(o(1)), \label{eq10GW}\\
\mathbb{P}(S_{n} = n-1) & =\frac{h}{a_{n}} g_{Y}(0) \exp(o(1)). \label{eq11GW}
\end{align}
Hence, 
\eqref{eq4GW} yields \eqref{eq7GW} in this case too. Then, the proof of \ref{ReGW1} and \ref{ReGW2} in case
\ref{case2GW} is completed as above. 

Finally, we prove the case $\gl=0$ of \ref{ReGW1}. 
If $m_n=o(\sqrt n)$, then \eqref{eq3GW} holds for $q_n=q$ fixed, and thus
\eqref{eq8GW} holds by the proof above, which yields the conclusion.

It remains to consider large $m_n$, say $m_n\ge n^{1/3}$. 
We consider again only 
$n\equiv 1\pmod h$, since otherwise $\cT_{\bp,n}$ does not exist.
By the local limit theorem \eqref{eq6GW} or \eqref{eq11GW},
it follows that 
(with $a_n=n\qq$ in case \ref{case1GW}),
for some $c>0$,
\begin{align}\label{lc2}
  \P(S_n=n-1) \sim \frac{c}{a_n}. 
\end{align}
Since $a_n$ is regularly varying with exponent $1/\ga<1$, 
\eqref{lc2} implies that
for large $n$,
\begin{align}\label{lc3}
\P(S_n=n-1) > n^{-1}.  
\end{align}
It follows from \eqref{eq2GW}
(with $q=1$), together with \eqref{lc3}, that
\begin{align}  
\E N_{T_{n}}(\mathcal{T}_{\mathbf{p},n}) 
=  n\pi_{\mathbf{p}}(T_{n}) 
 \frac{\mathbb{P}(S_{n-m_{n}} = n-m_{n})}{\mathbb{P}(S_{n} = n-1)} 
\le n^2 \pi_{\mathbf{p}}(T_{n})
\le n^2 (\max_i p_i)^{m_n}
\to0,
\end{align}
since  $\max_i p_i<1$ and we assumed $m_n\ge n^{1/3}$.
We immediately obtain also convergence of higher moments:
\begin{align}  
\E \bigsqpar{N_{T_{n}}(\mathcal{T}_{\mathbf{p},n})^q} 
\le n^{q-1}\E \bigsqpar{N_{T_{n}}(\mathcal{T}_{\mathbf{p},n})} 
\le n^{q+1} (\max_i p_i)^{m_n}
\le n^{q+1} (\max_i p_i)^{n^{1/3}}
\to0.
\end{align}
\end{proof}

We can from \refT{ThPo2} obtain similar results for $N_\bmk(\cT_{\bp,n})$,
the number of fringe trees with a given degree statistic $\bmk$, but we
leave the details to the reader.

\appendix
\section{A local limit theorem for drawing without replacement}\label{ALLT}

We used above a local limit theorem for the sum of a number of values
obtained by drawing without replacement.
We consider here the situation in somewhat greater generality.
Let $\bd=(d_1,\dots,d_n)$ be a given sequence of real numbers, and let $S_m$ 
($0\le m\le n$) 
be the sum of $m$ of them, obtained by drawing without replacement.
Formally, we may take a uniformly random permutation $\tau$ of $\setn$
and define
\begin{align}\label{Sm}
S_m=S_m(\bd):=\sum_{i=1}^m d_{\tau(i)}.
\end{align}

It is well known that for large $m$ and $n-m$, under suitable conditions on
$\bd$, the distribution of $S_m$ is asymptotically normal.
To state the result formally, we suppose that we are given a sequence
$\bd_\kk=(d_{i\kk})_{i=1}^{n_\kk}$, $\kk\ge1$, of sequences of real numbers, 
where $(n_\kk)_{\kk \geq 1}$ is some sequence of natural numbers;
we also suppose that $(m_\kk)_{\kk \geq 1}$ is another given sequence with $0\le m_\kk\le
n_\kk$ for every $\kk \geq 1$, and we consider the random sums
$\Smk=\Smk(\bd_\kk)$. 
(We simplify the notation, since there is no risk of confusion.)
We define 
\begin{align}\label{apa}
  \bxd_\kk&:=\frac{1}{n_\kk}\sum_{i=1}^{n_k} d_{i\kk},
\\
Q_\kk&:=\sum_{i=1}^{n_k}(d_{i\kk}-\bxd_\kk)^2,\label{apb}
\\
\hgssk&:=\frac{m_\kk}{n_\kk}\lrpar{1-\frac{m_\kk}{n_\kk}}Q_\kk.\label{apc}
\end{align}
It is easily seen that
\begin{align}\label{apd}
  \E \Smk = m_\kk\bxd_\kk
\qquad\text{and}\qquad 
\Var\Smk = \frac{m_\kk(n_\kk-m_\kk)}{n_\kk(n_\kk-1)}Q_\kk
=\frac{n_\kk}{n_\kk-1}\hgssk.
\end{align}
(We use $\hgssk$ as a convenient approximation to $\Var\Smk$; in the
asymptotic results below it does not matter which one we use.)
Note also that
if $D_\kk$ is a random element of the sequence $\bd_\kk$, then 
\begin{align}\label{ape}
\E D_\kk=\bxd_\kk
\qquad\text{and}\qquad
\Var D_\kk=Q_\kk/n_\kk
.\end{align}
(This is the case $m_\kk=1$ of \eqref{apd}.)

To avoid trivialities, we assume $Q_\kk>0$; otherwise, $d_{i\kk}=\bxd_\kk$
for all $i$, and $\Smk$ is non-random.
We assume the following Lindeberg type condition:
\begin{align} \label{Lindeberg}
\frac{1}{Q_\kk}\sum_{|d_{i\kk}-\bxd_\kk|>\eps\hgsk}(d_{i\kk}-\bxd_\kk)^2\to0
\qquad \text{as $\kktoo$, for every $\eps>0$}.
\end{align}
Then the following central limit theorem was proved by Erd\H os and R\'enyi \cite{ER1959},
see also H\'ajek \cite{Hajek} for another (simpler) proof (and proof of necessity
of \eqref{Lindeberg}), and H\'ajek \cite{Hajek-rejective}
 and Ji{\v r}ina \cite{Jirina} for  generalizations.

\begin{theorem}[\cite{ER1959}]\label{TA1}
With notations as above, if \eqref{Lindeberg} holds, then, as \kktoo,
\begin{align}\label{ta1}
  \frac{\Smk-\E\Smk}{\hgsk}
=   \frac{\Smk-m_\kk\bxd_\kk}{\hgsk}
\dto \N(0,1).
\end{align}
\end{theorem}

Further known results are
a functional limit theorem \cite{Rosen} and
Berry--Esseen estimates on the rate of convergence 
\cite{Hoglund1976,Hoglund1978}.
Our purpose is to show a corresponding local limit theorem,
in the case when all $d_i$ are integers.
We guess that also such a local limit theorem has been proved earlier, but
we have 
failed to find any reference except \cite{Benedicks-exjobb}, which we have
been unable to obtain, and we therefore give a rather general theorem
with a detailed proof.

Given a sequence $\bd=(d_1,\dots,d_n)$ of integers,
we define
$|\bd|:=n$ and,
in analogy with \refSS{SStrees}, the  statistic
$\bn_\bd=(n_\bd(i))_{i\in\bbZ}$, where
\begin{align}\label{ac7}
  n_\bd(k):=\numset{1\le i\le |\bd|:d_i=k}
.\end{align}
We define also
\begin{align}\label{ac8}
  p_k(\bd):=\frac{n_\bd(k)}{|\bd|},\qquad k\in\bbZ;
\end{align}
thus $(p_k(\bd))_{k\in\bbZ}$ is a probability distribution, namely,
the empirical distribution of $(d_i)_{i=1}^n$.
We assume a condition similar to (and slightly weaker than) 
\refConds{Condition1} and \ref{Cond2}:
\begin{condition} \label{Cond1Z}
$\bd_{\kappa} = (d_{i\kappa})_{i=1}^{n_\kk}$, $\kk\ge1$, 
are sequences  such that, 
as $\kappa \rightarrow \infty$:
\begin{enumerate}[label=\upshape(\roman*)]
\item\label{Cond1Za} 
$|\bd_{\kappa}| \to \infty$, 
\item \label{Cond1Zb}
For every $k\in\bbZ$, we have
$p_{k}(\bd_{\kappa}) \rightarrow p_{k}$, 
where $\bp := (p_{k})_{k\in\bbZ}$ is a 
probability distribution on $\bbZ$.
\item \label{Cond1Zc}
$Q_\kk=O(n_\kk)$.
In other words, if $D_\kk$ is a random element of the sequence $\bd_\kk$,
then $\Var D_\kk=O(1)$.
\end{enumerate} 
\end{condition}

We can now state our local limit theorem.

\begin{theorem}\label{TLLT}
  Let $\bd_\kk$ ($\kk\ge1$) be (finite) integer sequences such that,
with notations as above,
\refCond{Cond1Z} holds with limit $\bp$ nonlattice,
and also the Lindeberg condition \eqref{Lindeberg} holds.
Then, as $\kktoo$, uniformly for all $k\in\bbZ$,
\begin{align}\label{tllt}
  \P(\Smk=k)
=\frac{1}{\sqrt{2\pi\hgssk}}
\lrpar{ \exp\lrpar{-\frac{(k-m_\kk\bxd_\kk)^2}{2\hgssk}}+o(1)}
.\end{align}
\end{theorem}

The proof uses the standard method of Fourier inversion using a global limit
theorem (\refT{TA1}) combined with an estimate of the characteristic
function away from 0 (and multiples of $2\pi$).
Note that both the proof of \refT{TA1} in \cite{ER1959} and the
Berry--Esseen estimates in \cite{Hoglund1976,Hoglund1978} are based on
estimates of the characteristic function, but they need only estimates for
close to 0 (on the other hand, they need more precise estimates there).
Note also that the assumption that $\bp$ is nonlattice is necessary,
since \eqref{tllt} cannot hold without modification if, for example, all
$d_{i\kk}$ are even.
We prove first a lemma with the required estimate 
of the characteristic function $\gf_{\Smk}(t)$ of $\Smk$.

\begin{lemma}\label{L1}
Suppose that \refCond{Cond1Z}\ref{Cond1Zb}  holds, and that $\bp$ is nonlattice.
Then there exist  constants $C>0$ and $c>0$ such that, for all $\kk$
and all $m_\kk$ with $0\le m_\kk\le n_\kk$,
\begin{align}\label{l1}
  |\gf_{\Smk}(t)| \le C e^{-c m_\kk (1-m_\kk/n_\kk)t^2},
\qquad |t|\le\pi.
\end{align}
\end{lemma}

\begin{proof}
For convenience, we drop the subscripts $\kk$.
We use the conditioning method implicit in \cite{ER1959}
(see also \cite{Hajek} and \cite[Example 2.3]{Holst:conditional}).
Let $\xi_i$ ($1\le i\le n$) be \iid{} Bernoulli random variables with
\begin{align}
  \label{aa0}
\P(\xi_i=1)=\px:=m/n,
\end{align}
let $\eta_i:=d_i\xi_i$, 
and define the random vectors $\zeta_i:=(\xi_i,\eta_i)=(\xi_i,d_i \xi_i)$.
Let
\begin{align}\label{aa1}
  Z=(X,Y):=\sumin \zeta_i.
\end{align}
Note that $X=\sumin \xi_i$ has a binomial distribution $\Bi(n,\px)$.
Moreover, it follows from the construction that
\begin{align}\label{aa2}
  S_m\eqd (Y\mid X=m).
\end{align}
Let $\gf_Z(s,t):=\E e^{\ii s X+\ii t Y}$ be the characteristic function of
$Z$ (for $s,t\in\bbR$).
Then it follows from \eqref{aa2} that the characteristic function of $S_m$ is,
see \cite{ER1959} and \cite{Holst:conditional},
\begin{align}\label{aa3}
  \gfsm(t)
=\frac{\E [e^{\ii t Y}\indic{X=m}]}{\P(X=m)}
=\frac{1}{2\pi\P(X=m)}\intpipi e^{-\ii m s}\gf_Z(s,t)\dd s.
\end{align}

Since \eqref{aa1} is a sum of independent variables, we have
\begin{align}\label{aa4}
  \gf_Z(s,t)=\prodin \E e^{\ii (s \xi_i+ t \eta_i)}
=\prodin (1-\px+\px e^{\ii(s+td_i)})
=\prod_{d\in\bbZ} (1-\px+\px e^{\ii(s+ dt)})^{n_\bd(d)}.
\end{align}
For any $u \in \mathbb{R}$,
\begin{align}\label{aa5}
  \bigabs{1-\px+\px e^{\ii u}}^2&
=(1-\px)^2+\px^2+2\px(1-\px)\Re e^{\ii u}
=1-2\px(1-\px)(1-\cos u)
\notag\\&
\le \exp\bigpar{-2\px(1-\px)(1-\cos u)}
.\end{align}
Consequently, \eqref{aa4} yields
\begin{align}\label{aa6}
  |\gf_Z(s,t)|
\le  \exp\Bigpar{-\px(1-\px)\sum_{d \in \mathbb{Z}} n_\bd(d)\bigpar{1-\cos (s+td)}}.
\end{align}

By assumption, 
the limit distribution
$\bp$ is nonlattice, and thus the set $\set{i-j:i,j\in\supp(\bp)}$
is not contained in any proper subgroup of $\bbZ$; 
this means that the set generates  $\bbZ$, and thus there exist  finite
sequences $(i_k)_{k=1}^K$  and $(j_k)_{k=1}^K$ in $\supp(\bp)$ and integers
$\nu_k$ such that 
\begin{align}\label{aa7}
1=\sum_{k=1}^K \nu_k(i_k-j_k).  
\end{align}
Let $\Do:=\set{i_k,j_k:1\le k\le K}\subseteq\supp(\bp)$.
Since $\Do$ is finite, we have for all sufficiently large $\kk$, by
\eqref{ac8} and \refCond{Cond1Z},
\begin{align}\label{aa8}
\frac{  n_{\bd}(d)}{n}
=\frac{  n_{\bd}(d)}{|\bd|} 
\ge \frac12 p_d,
\qquad d\in \Do.
\end{align}
We consider in the sequel only such $\kk$, and note that this suffices,
since for any given $c$, \eqref{l1} trivially holds for any fixed $\kk$ if
we choose $C$ large enough. Then, \eqref{aa6} implies, 
since $1-\cos(s+td)\ge0$ for all $s,t \in \mathbb{R}$ and $d \in \bbZ$,
\begin{align}\label{aa9}
  |\gf_Z(s,t)|
\le  \exp\Bigpar{-\tfrac12\px(1-\px)n
\sum_{d\in \Do} p_d\bigpar{1-\cos (s+td)}}.
\end{align}

\ccreset\CCreset
It remains to study the function
\begin{align}\label{aa10}
g(s,t):= \sum_{d\in \Do} p_d\bigpar{1-\cos (s+td)}.
\end{align}
First, for some $\ccname\cca>0$, if $|s|,|t|<\cca$, then
$\cos(s+td)\le 1-\frac13(s+td)^2$ for $d\in\Do$ and thus
\begin{align}\label{aa12}
g(s,t)\ge h(s,t):=\tfrac13\sum_{d\in \Do} p_d (s+td)^2.
\end{align}
The \rhs{} $h(s,t)$ is a homogeneous polynomial in $s$ and $t$.
Assume $h(s,t)=0$. Then,
$s+td=0$ for every $d\in \Do$.
In particular, $s+ti_k=s+tj_k=0$ for every $i_k$ and $j_k$ in \eqref{aa7};
hence, $t(i_k-j_k)=0$, and \eqref{aa7} implies
$t=\sum_{k=1}^K \nu_k(i_k-j_k)t=0$. Hence also $s=0$.
Consequently, $h(s,t)=0\iff (s,t)=(0,0)$. By compactness, $h(s,t)$ has a
positive minimum value $\ccname\ccb>0$ on the unit circle $\set{(s,t) \in \mathbb{R}^{2}: s^2+t^2=1}$,
and thus by homogeneity
  \begin{align}\label{aa13}
    h(s,t)\ge \ccb(s^2+t^2),
\qquad s,t\in\bbR.
  \end{align}
Similarly, by \eqref{aa10}, if $g(s,t)=0$, then $s+td\equiv 0\pmod{2\pi}$ for
every $d\in \Do$. Arguing as above but modulo $2\pi$ (i.e., in the group
$\bbR/2\pi\bbZ$), we find that 
\begin{align}
 g(s,t)=0\iff s\equiv t\equiv 0\pmod{2\pi}.
\end{align}
In particular, $g(s,t)>0$ on $[-\pi,\pi]^2\setminus(-\cca,\cca)^2$,
and thus, by compactness, 
\begin{align}\label{aa11}
g(s,t)\ge \ccname\ccq,
\qquad 
(s,t)\in [-\pi,\pi]^2\setminus(-\cca,\cca)^2,  
\end{align}
for some $\ccq>0$.
Taking $\ccname\ccd:=\min\bigpar{\ccb,\ccq/(2\pi^2)}$, it follows 
by \eqref{aa12} and \eqref{aa13} for $|s|,|t|< \cca$, and 
\eqref{aa11} otherwise, that
\begin{align}\label{aa22}
  g(s,t)\ge \ccd(s^2+t^2),
\qquad |s|,|t|\le \pi.
\end{align}

We may for convenience assume $p=m/n\le\frac12$, since otherwise
we may replace $m$ by $n-m$ by interchanging the drawn set
of values $\set{d_{\tau(i)}:i\le m}$ and its complement.
We may also assume $m\ge1$, since \eqref{l1} is trivial for $m=0$.

Then \eqref{aa9}, \eqref{aa10}, and \eqref{aa22} yield,
recalling \eqref{aa0}, 
\begin{align}\label{aa99}
  |\gf_Z(s,t)|
\le  \exp\bigpar{- \ccname\cce\px n(s^2+t^2)}
=\exp\bigpar{- \cce m(s^2+t^2)},
\qquad |s|,|t|\le \pi,
\end{align}
for some constant $\cce >0$. Hence, integrating over $s$,
it follows from \eqref{aa3} that there exists $\CC >0$ such that
\begin{align}\label{aa93}
 \bigabs{ \gfsm(t)}
\le\frac{1}{\P(X=m)}\intpipi \bigabs{\gf_Z(s,t)}\dd s
\le\frac{\CCx}{\P(X=m)\sqrt{m}}
 \exp\bigpar{- \cce mt^2},
\qquad |t|\le\pi
.\end{align}
Furthermore, $X\simx\Bi(n,p)$ and $m=np$ with $1\le m\le n/2$, and thus,
as is well known, $\P(X=m)\ge \cc (np)\qqw=\ccx m\qqw$ for some $\ccx>0$
by a direct calculation using Stirling's formula.
Hence, 
the result \eqref{l1} follows from \eqref{aa93}.
\end{proof}

\begin{proof}[Proof of \refT{TLLT}]
  As said above, this uses a standard argument.

Let $\muk:=\E\Smk=m_\kk\bxd_\kk$ and
$Y_\kk:=(\Smk-\mu_\kk)/\hgsk$.
Then \refT{TA1} says $Y_\kk\dto \N(0,1)$, and thus the \chf{}
\begin{align}\label{ll1}
  \gf_{\Yk}(t)\to e^{-t^2/2},
\end{align}
for every fixed $t\in\bbR$.
(In fact, \refT{TA1} was proved in \cite{ER1959} by showing \eqref{ll1}.)
Furthermore,
\begin{align}\label{ll2}
  \gf_\Yk(t)=\gf_{\Smk}\parfrac{t}{\hgsk}e^{-\ii(\muk/\hgsk)t}.
\end{align}

By Fourier inversion and \eqref{ll2},
for every $k\in\bbZ$,
\begin{align}\label{ll3}
  \P(\Smk=k)=\frac{1}{2\pi}\intpipi e^{-\ii tk}\gf_{\Smk}(t)\dd t
=\frac{1}{2\pi \hgsk}\int_{-\hgsk\pi}^{\hgsk\pi} 
e^{-\ii tk/\hgsk}\gf_\Yk(t)e^{\ii(\muk/\hgsk)t}\dd t.
\end{align}
Hence, subtracting a well known integral,
\begin{align}\label{ll4}
& 2\pi\hgsk\P(\Smk=k)-\sqrt{2\pi}e^{-(k-\muk)^2/(2\hgssk)} \nonumber \\
& \qquad \qquad =
\intoooo
e^{-\ii t(k-\muk)/\hgsk}\Bigpar{\gf_\Yk(t)\indic{|t|\le\hgsk\pi}-e^{-t^2/2}}\dd t.
\end{align}
Dividing by $\sqrt{2\pi}$ and taking absolute values yield
\begin{align}\label{ll5}
& \bigabs{\sqrt{2\pi\hgssk}\P(\Smk=k)-e^{-(k-\muk)^2/(2\hgssk)}} \nonumber \\
& \quad \quad \quad \le
\frac{1}{\sqrt{2\pi}}
\intoooo
\Bigabs{\gf_\Yk(t)\indic{|t|\le\hgsk\pi}-e^{-t^2/2}}\dd t.
\end{align}
The \rhs{} is independent of $k$, so we may take the supremum over all $k$.
Hence, it remains only to show that the integral in \eqref{ll5} tends to 0
as \kktoo.
To see this, note first that
since $\Smk$ is integer valued, it follows from \eqref{ta1}
that necessarily $\hgsk\to\infty$. (This may also be shown directly from
\eqref{Lindeberg} and the assumption that all $d_{i\kk}$ are integers.)
Consequently, \eqref{ll1} implies that
the integrand in \eqref{ll5} tends to 0 as \kktoo, for every
fixed $t$. 
Furthermore, \eqref{ll2} and \refL{L1} show that, for every $t\in\bbR$,
\begin{align}\label{ll6}
\bigabs{\gf_\Yk(t)\indic{|t|\le\hgsk\pi}}
&=
\Bigabs{\gf_{\Smk}\parfrac{t}{\hgsk}}\indic{|t/\hgsk|\le\pi} \notag
\\
& \le 
Ce^{-c m_\kk (1-m_\kk/n_\kk)t^2/\hgssk}
=
Ce^{-c n_\kk t^2/Q_\kk}.
\end{align}
Hence,
\refCond{Cond1Z}\ref{Cond1Zc}
implies that the integrand in \eqref{ll5} is uniformly bounded by
$Ce^{-c t^2}$, and thus dominated convergence applies.
Consequently, the integral in \eqref{ll5} tends to 0 as \kktoo, 
which completes the proof.
\end{proof}

\begin{remark}
  We used for simplicity compactness to obtain \eqref{aa11} and thus
the existence of $c>0$ in
  \eqref{l1}. Alternatively, one might obtain an explicit constant $c$
  (depending on $\bp$ and for all sufficiently large $\bn$)
by estimating $g(s,t)$ (or the corresponding sum over all $d$)
using \cite{Benedicks1975}.
\end{remark}

\begin{remark}\label{R1Z}
  Note that \refCond{Cond1Z}\ref{Cond1Zc} is not needed for the proof of
  \refL{L1}.
However, this condition (or something similar) is needed for \refT{TLLT},
as is shown by the following counterexample.
 Note that \refT{TA1} applies in this example.

Let $\bd_\kk$ consist of $n_\kk:=3\kk+2\floor{\kk\qq}$ elements:
each of $0, 1, -1$ appears $\kk$ times,
and each of $\pm2\floor{\kk^{3/4}}$ appears $\floor{\kk\qq}$ times.
Let (for simplicity) $m_\kk:=\floor{n_\kk/2}$.
It is easily seen that $Q_\kk\sim 8\kk^2$ and $\hgssk\sim 2\kk^2$, and that
\eqref{Lindeberg} holds (trivially). 
Also \refCond{Cond1Z}\ref{Cond1Za}--\ref{Cond1Zb} hold, with
$\supp{\bp}=\set{-1,0,1}$ and thus $\bp$ is nonlattice.
However, \refCond{Cond1Z}\ref{Cond1Zc} does not hold.

Let $Z_\kk:=\sum_{i=1}^{m_\kk}d_{\tau(i)}\indic{|d_{\tau(i)}|=1}$
be the sum of the drawn values that are $\pm1$. Then
$\Smk\equiv Z_\kk \pmod{2\floor{\kk^{3/4}}}$. 
Hence,
if $\Smk=\floor{\kk^{3/4}}$, then 
$|Z_\kk|\ge\floor{\kk^{3/4}}$. 
It follows, by conditioning on the number of drawn values
that are $\pm1$ 
and using Chernoff bounds for the hypergeometric distribution
(see e.g.\ \cite[Theorem 2.10]{JLR})
that $\P\bigpar{\Smk=\floor{\kk^{3/4}}}\le e^{-c\kk\qq}$, which shows
that \eqref{tllt} does not hold for $k=\floor{\kk^{3/4}}$.
\end{remark}

\section{The constant in Theorem \ref{Tstein}}\label{ConstStein} 

In this appendix, we show that the constant $C$ in Theorem \ref{Tstein} can be chosen as $60 e^{\frac{29}{4}}$ (see Remark \ref{RemarkSteinConst}). We will use the notation defined in the proof of Theorem \ref{Tstein}. 

As shown at the beginning of the proof of Theorem \ref{Tstein}, by \eqref{jm3}, we may assume that $\Delta$ defined in \eqref{jm2} satisfies $\Delta < \frac{1}{4} \lambda$.

Note that for $x \geq 1$ and $0 \leq k \leq x/2$ integers,
\begin{align}
\frac{(x)_{k}}{x^{k}} = \prod_{i=0}^{k-1} \frac{x-i}{x} = \exp \left( \sum_{i=0}^{k-1} \ln \left(1 - \frac{i}{x} \right) \right).
\end{align}
Note also that, for $i = 0, \dots, k-1$,
\begin{align}
-\frac{2i}{x} \leq \ln \left(1 - \frac{i}{x} \right) \leq -\frac{i}{x}.
\end{align}
Then,
\begin{align} \label{ExtIneq1}
\exp \left( -\frac{k(k-1)}{x} \right) \leq \frac{(x)_{k}}{x^{k}} \leq \exp \left( -\frac{k(k-1)}{2x} \right).
\end{align}
\noindent It follows from \eqref{pip}, \eqref{eq4G}, and \eqref{ExtIneq1} that
\begin{align} \label{ExtIneq2}
\mathbb{E} N_{T}(\cTbn) &\leq n\pi_{\bp(\bn)}(T) \exp \left( \frac{m(m-1)}{n} - \sum_{i\in \cD(T)} \frac{n_{T}(i)(n_{T}(i) -1)}{2n(i)} \right), \\
\label{ExtIneq3}
\mathbb{E} N_{T}(\cTbn) &\geq n\pi_{\bp(\bn)}(T) \exp \left( \frac{m(m-1)}{2n} - \sum_{i\in \cD(T)} \frac{n_{T}(i)(n_{T}(i) -1)}{n(i)} \right),
\end{align}
and
\begin{align} \label{ExtIneq4}
\mathbb{E} (N_{T}(\cTbn))_{2} \leq (n\pi_{\bp(\bn)}(T))^{2} \exp \left( \frac{2(2m-1)(m-1)}{n} - \sum_{i\in \cD(T)} \frac{n_{T}(i)(2n_{T}(i) -1)}{n(i)} \right).
\end{align}
By \eqref{ExtIneq3} and \eqref{ExtIneq4}, 
recalling $N_{T}(\cTbn)=W$
and using $e^x-e^y \le (x-y)e^x$ for real $x$ and $y$,
\begin{align}
&\operatorname{Var} W - \mathbb{E} W 
=  \mathbb{E}[(W)_2] - (\mathbb{E} W)^2 \nonumber \\
&\qquad\leq \frac{3m^{2}}{n}(n\pi_{\bp(\bn)}(T))^{2} \exp \left( \frac{2(2m-1)(m-1)}{n} - \sum_{i\in \cD(T)} \frac{n_{T}(i)(2n_{T}(i) -1)}{n(i)} \right),
\end{align}
and thus, by \eqref{ExtIneq3},
\begin{align}
\frac{\operatorname{Var} W - \mathbb{E} W}{\mathbb{E} W} 
&\leq \frac{3m^{2}}{n} n\pi_{\bp(\bn)}(T) \exp \left( \frac{(7m-4)(m-1)}{2n} - \sum_{i\in \cD(T)} \frac{(n_{T}(i))^{2}}{n(i)} \right) \notag\\
&\leq \frac{3m^{2}}{n} n\pi_{\bp(\bn)}(T) \exp \left( \frac{7m^{2}}{2n} \right).
\end{align}
Then, \eqref{tau7} and \eqref{tau8} imply that
\begin{align} \label{ExtIneq5}
\dtv\left(N_{T}(\cTbn),\Po(\gl)\right)
\leq 36\gD + \frac{24 \gD }{\lambda} n\pi_{\bp(\bn)}(T) \exp\left(\frac{28 \gD}{\lambda}\right).
\end{align}
On the other hand, by \eqref{ExtIneq3} and \eqref{jm2},
\begin{align}\label{ExtIneq6}
\lambda \exp\left(\frac{\gD}{\lambda}\right) \geq n\pi_{\bp(\bn)}(T).
\end{align}
Then, we deduce from \eqref{ExtIneq5}   and \eqref{ExtIneq6}
that
\begin{align}
\dtv\left(N_{T}(\cTbn),\Po(\gl)\right) 
\leq 36\gD +  \frac{24\gD}{\lambda} \gl \exp\left(\frac{29\gD}{\lambda}\right)
\leq 60  \gD \exp\left(\frac{29\gD}{\lambda}\right).
\end{align}
Furthermore, since we assumed $\gD < \lambda /4$, it follows that 
\begin{align}
\dtv\left(N_{T}(\cTbn),\Po(\gl)\right) \leq 60 e^{\frac{29}{4}} \gD.
\end{align}

\begin{acks}
We are grateful to the reviewers for their helpful comments and corrections. We especially thank the reviewer whose comment led to Remark \ref{RemarkSteinConst} concerning the computation of the constant in Theorem \ref{Tstein}. 
\end{acks}

\newcommand\AAP{\emph{Adv. Appl. Probab.} }
\newcommand\JAP{\emph{J. Appl. Probab.} }
\newcommand\JAMS{\emph{J. \AMS} }
\newcommand\MAMS{\emph{Memoirs \AMS} }
\newcommand\PAMS{\emph{Proc. \AMS} }
\newcommand\TAMS{\emph{Trans. \AMS} }
\newcommand\AnnMS{\emph{Ann. Math. Statist.} }
\newcommand\AnnPr{\emph{Ann. Probab.} }
\newcommand\CPC{\emph{Combin. Probab. Comput.} }
\newcommand\JMAA{\emph{J. Math. Anal. Appl.} }
\newcommand\RSA{\emph{Random Structures Algorithms} }
\newcommand\DMTCS{\jour{Discr. Math. Theor. Comput. Sci.} }

\newcommand\AMS{Amer. Math. Soc.}
\newcommand\Springer{Springer-Verlag}
\newcommand\Wiley{Wiley}

\newcommand\vol{\textbf}
\newcommand\jour{\emph}
\newcommand\book{\emph}
\newcommand\inbook{\emph}
\def\no#1#2,{\unskip#2, no. #1,} 
\newcommand\toappear{\unskip, to appear}

\newcommand\arxiv[1]{\texttt{arXiv}:#1}
\newcommand\arXiv{\arxiv}

\newcommand\xand{and }
\renewcommand\xand{\& }

\def\nobibitem#1\par{}

\end{document}